\documentclass[12pt]{amsart}

\usepackage{environ, xcolor}
\NewEnviron{commentA}{}

\newcommand\Aoff{\RenewEnviron{commentA}{}}
\Aoff{} 

\usepackage[normalem]{ulem}

\usepackage{amsmath, amssymb}
\usepackage{array}
\usepackage[frame,cmtip,arrow,matrix,line,graph,curve]{xy}
\usepackage{graphpap, color, paralist, pstricks}
\usepackage[mathscr]{eucal}
\usepackage[pdftex]{graphicx}
\usepackage[pdftex,colorlinks,backref=page,citecolor=blue]{hyperref}
\usepackage{cleveref}
\usepackage{tikz-cd, verbatim, transparent}

\newtheorem{theorem}{Theorem}[section]
\newtheorem{proposition}[theorem]{Proposition}
\newtheorem{corollary}[theorem]{Corollary}
\newtheorem{lemma}[theorem]{Lemma}

\theoremstyle{definition}
\newtheorem{definition}[theorem]{Definition}
\newtheorem{example}[theorem]{Example}

\newtheorem{remark}[theorem]{Remark}

\newtheorem{assumption}[theorem]{Assumptions}

\newtheorem*{acknowledgement}{Acknowledgments}

\numberwithin{equation}{section}

\usepackage[margin=1in]{geometry} 
\usepackage{amsmath,amssymb,mathtools, mathrsfs,esint,tikz-cd,verbatim}

\newcommand{\Z}{{\mathbb Z}}
\newcommand{\Q}{{\mathbb Q}}

\newcommand{\C}{{\mathbb C}}

\newcommand{\G}{{\mathbb G}}
\newcommand{\bP}{{\mathbb P}}
\newcommand{\bA}{{\mathbb A}}
\newcommand{\calO}{{\mathcal O}}
\newcommand{\calF}{{\mathcal F}}
\newcommand{\calR}{{\mathcal R}}
\newcommand{\calQ}{{\mathcal Q}}
\newcommand{\calL}{{\mathcal L}}
\newcommand{\on}[1]{\operatorname{#1}}

\newcommand{\SHom}{{\on{\mathcal{H}om}}}
\newcommand{\SExt}{{\on{\mathcal{E}xt}}}

\newcommand{\im}{{\on{im}}}
\newcommand{\Aut}{{\on{Aut}}}

\newcommand{\Spec}{{\on{Spec}}}
\newcommand{\Proj}{{\on{Proj }}}

\newcommand{\Ext}{{\on{Ext}}}

\newcommand{\Pic}{{\on{Pic}}}

\newcommand{\Cl}{\on{Cl}}

\newcommand{\bC}{\mathbb{C}}

\newcommand{\bQ}{\mathbb{Q}}

\newcommand{\bG}{\mathbb{G}}

\newcommand{\calX}{\mathcal{X}}

\newcommand{\calK}{\mathcal{K}}

\newcommand{\Hilb}{\mathrm{Hilb}}

\newcommand{\calD}{\mathcal{D}}

\newcommand{\calM}{\mathcal{M}}
\newcommand{\calN}{\mathcal{N}}

\newcommand{\calS}{\mathcal{S}}

\newcommand{\bZ}{\mathbb{Z}}

\newcommand{\cX}{\mathcal X}

\newcommand{\Supp}{\textrm{Supp}}

\newcommand{\cL}{\mathcal L}

\newcommand{\calC}{\mathcal C}

\DeclareFontFamily{U}{mathb}{\hyphenchar\font45}
\DeclareFontShape{U}{mathb}{m}{n}{
      <5> <6> <7> <8> <9> <10> gen * mathb
      <10.95> mathb10 <12> <14.4> <17.28> <20.74> <24.88> mathb12
      }{}
\DeclareSymbolFont{mathb}{U}{mathb}{m}{n}
\DeclareMathSymbol{\righttoleftarrow}{3}{mathb}{"FD}

\subjclass[2020]{}

\title{Weighted Projective Degenerations of $\mathbb{P}^{n}$}

\author[K. DeVleming]{Kristin DeVleming} 
\address{Department of Mathematics, University of California San Diego, 2985 Muir Lane, La Jolla, CA 92093, USA}
\email{kedevleming@ucsd.edu}

\author[J. Li]{Jennifer Li}
\address{Department of Mathematics, Princeton University, Fine Hall, 304 Washington Road, Princeton, NJ 08540, USA}
\email{jenniferli@princeton.edu}

\author[S. Torres]{Sebasti\'{a}n Torres}
\address{Department of Mathematics, Universidad de Concepción, Av. Esteba Iturra s/n, CFM Of. 522, Concepción, Chile}
\email{storresk(at)udec.cl}

\begin{document}

\begin{abstract}
We study weighted projective klt $\mathbb{Q}$-Gorenstein degenerations of projective space $\bP^n$ and construct infinitely many new degenerations of projective space in any dimension.  We also study in detail the deformations of arbitrary weighted projective threefolds.  The rest of the paper provides several applications of the existence of degenerations of $\bP^n$ and our methods apply to find many degenerations of $\bP^n$ beyond just weighted projective spaces. 
\end{abstract}

\maketitle
\setcounter{tocdepth}{1}
\tableofcontents

\section{Introduction}

The aim of this paper is to study mildly singular varieties that appear as \textit{degenerations} or \textit{limits} of projective space $\bP^n_{\bC}$.
We say that a Fano variety $X$ is a $\bQ$-Gorenstein Fano degeneration of $\bP^n$ if there exists a flat family $\cX \to B$, where $0 \in B$ is a reduced pointed scheme, such that $K_{\cX/B}$ is $\bQ$-Cartier,
with fibers
$\cX_0 \cong X$, and $\cX_b \cong \bP^n$ for general $b \in B$.  In what follows, a `degeneration' of $\bP^n$ will always refer to a $\bQ$-Gorenstein Fano degeneration. 
We are particularly interested in the case when $X$ has (at worst) klt singularities.
Such degenerations appear naturally in many areas of birational geometry and moduli and have a particular importance in K-stability and moduli of Fano varieties.  

In dimension one, klt is equivalent to smooth and therefore by flatness, the classification of klt degenerations of $\bP^1$ is equivalent to the classification of smooth genus $0$ curves.  In particular, the only klt degeneration of $\bP^1$ is $\bP^1$ itself.

The situation in dimension two is also completely understood. In \cite{M91}, Manetti classified all normal degenerations of $\mathbb{P}^{2}$ that have quotient singularities, which are precisely the klt singularities in dimension 2.  Later, Hacking and Prokhorov in \cite{HP10} refined the classification of Manetti and demonstrated that every $\mathbb{Q}$-Gorenstein klt Fano degeneration of $\mathbb{P}^{2}$ is either a weighted projective surface of form $\mathbb{P}(a^{2}, b^{2}, c^{2})$ for positive integers $a, b, c$ that satisfy the Markov equation $a^{2} + b^{2} + c^{2} = 3abc$, or a partial smoothing of such a weighted projective surface.  Precisely, they prove: 

\begin{theorem}
{\cite[Corollary 1.2]{HP10}}
\label{HP:markov}
Let $X$ be a klt
degeneration of $\mathbb{P}^{2}$.
Then, $X$ is one of the following: 
\begin{enumerate}
    \item $X \cong \mathbb{P}(a^{2}, b^{2}, c^{2})$ where $a, b$, and $c$ satisfy the Markov equation:
\begin{center}
$a^{2} + b^{2} + c^{2} = 3abc$, or
\end{center}
\item $X$ is isomorphic to a partial smoothing of one of the surfaces in (1).  
\end{enumerate}
Moreover, all such varieties $X$ admit a $\mathbb{Q}$-Gorenstein smoothing to $\mathbb{P}^{2}$.
\end{theorem}

In higher dimensions, very few results are known.  In dimension three, the situation is partially understood: there is a complete classification of terminal (resp. canonical) degenerations of $\bP^3$, by the first author (resp. by work of Ascher, the first author, and Liu):

\begin{theorem} 
{\cite[Theorem 4.33]{DeVleming}}
    Let $X$ be a terminal degeneration of $\bP^3$.  Then, $X \cong \bP^3$.
\end{theorem}

\begin{theorem} 
{\cite[Theorem 1.5]{ADL23}}
If $X$ is a canonical degeneration of $\bP^3$, then $X$ is isomorphic to one of the following:
\begin{enumerate}
\item $\mathbb{P}^{3}$;
\item The projective anti-canonical cone over $\mathbb{P}^{1} \times \mathbb{P}^{1}$;
\item $\mathbb{P}(1, 1, 2, 4)$; or 
\item A Gorenstein non $\Q$-factorial $\mathbb{Q}$-Fano 3-fold with an isolated singular point, denoted $X_u$, constructed in \S 4.2 of \cite{ADL23}.
\end{enumerate}
\end{theorem}

\begin{remark}
Using different methods, H\"{o}ring and Peternell arrived at the same classification of canonical degenerations of $\bP^3$ in \cite{HoPe24}.
\end{remark}

Beyond the canonical case for $n = 3$, and in general for $n \ge 4$, there is very little known about klt degenerations of $\bP^n$.  

\begin{remark}
    While this paper will focus on klt singularities, by work of Ishii in \cite{I91}, there is also an understanding of certain non-klt degenerations with isolated singularities as cones over lower dimensional varieties. 
\end{remark}

In this paper, we approach the problem of classifying degenerations of projective space in two ways: first, by building degenerations of higher dimensional varieties out of degenerations of lower dimensional varieties via a cone construction, and secondly, through direct analysis of the deformation theory of weighted projective threefolds.  Motivated by the result of Hacking and Prokhorov which says every klt degeneration of $\bP^2$ is a partial smoothing of a weighted projective surface, we focus primarily on the classification of weighted projective degenerations of $\bP^n$.  Our main results in this direction are the existence of infinitely many weighted projective spaces arising as degenerations of $\bP^n$.

Using known degenerations of $\bP^2$ and $\bP^3$, we construct two infinite families of new klt $\mathbb{Q}$-Gorenstein weighted projective degenerations of $\mathbb{P}^{n}$ for any integer $n\geq 3$.

\begin{theorem}
\label{thm:infiniteFamiliesDegenerationsPn}
Let $X = \mathbb{P}(a_{0}, a_{1}, a_{2}, \dots , a_{n})$ be a well-formed projective space such that the weights $a_{i}$ are of the form (up to reordering)
\begin{enumerate}
\item $(a^{2}, b^{2}, abc_{n-2}, c^{2})$ for some $n \ge 2$, where $a^2 + b^2 + c^2 = 3abc$; or
\item $(a^{2}, b^{2}, 2c^{2}, 2abc_{n-3}, 4abc)$ for some $n \ge 3$, where $a^2 + b^2 + 2c^2 = 4abc$.
\end{enumerate}
Then $X$ admits a $\mathbb{Q}$-Gorenstein smoothing to $\mathbb{P}^{n}$.
\end{theorem}

Here a subscript $k$ indicates that the corresponding weight is repeated $k$ times in the list.
Note that the only tuple of weights $(a_0, \dots, a_n)$ that arises in both families is $(1,1,2_{n-2},4)$ (see, e.g. \cite[Lemma 4.60]{DeVleming}), so these in fact correspond to two distinct infinite families of degenerations of $\bP^n$. This result will follow from \Cref{Lemma:relativeProjectiveCone} and \Cref{lem:coneofdivisors}. 

We derive the previous results as consequences of the following more general theorem, which says given a $\bQ$-Gorenstein degeneration $X$ of $\bP^n$, an appropriate orbifold cone over $X$ is a $\bQ$-Gorenstein degeneration of $\bP^{n+1}$.

\begin{theorem}
\label{thm:cone}
Suppose $\mathbb{P}^{n}$ admits a $\mathbb{Q}$-Gorenstein degeneration to a Fano variety $X$. Denote by $L$ the $\mathbb{Q}$-line bundle on $X$ that is the limit of $\mathcal{O}(1)$, and suppose $L$ is $\bQ$-linearly equivalent to $\calO_X(-(n+1)K_X)$. Then $\mathbb{P}^{n+1}$ admits a $\mathbb{Q}$-Gorenstein degeneration to
\begin{center}
$C_p(X, L) := \Proj\:\displaystyle{\sum_{m \geq 0}} \biggl( \displaystyle{\sum_{r=0}^{m}} H^{0}(X, L^{[r]}) \cdot x_{n+1}^{m-r} \biggr).$
\end{center}
Furthermore, if $X$ is klt, then $C_p(X,L)$ is klt. 
\end{theorem}

This result has consequences for (local) smoothability of certain singularities: the vertex of the orbifold cone in the previous result is by construction smoothable, and this yields new examples of smoothable non-isolated cyclic quotient singularities.

We also introduce methods for detecting smoothability of weighted projective threefolds in general. Under mild assumptions on the weighted projective space (which, for example are satisfied for any Gorenstein weighted projective threefold), we prove that the smoothability of the weighted projective space can be detected purely by smoothability of its codimension 3 singular points. 

\begin{theorem}\label{introthm:defofthreefolds}
    Let $X = \bP(a_0,a_1,a_2,a_3)$ be a well-formed weighted projective threefold with coordinates $[x_0:x_1:x_2:x_3]$. Assume that $X$ has only $A_{n}$ singularities in codimension 2, and for each codimension 2 singular stratum, there is a section of $\omega_X^\vee$ that does not vanish at the generic point of the stratum.

    Denote by $U_i$ the locus of points where $x_i \ne 0$.  Then,
        \begin{enumerate}
            \item the first order deformation space $\Ext^1(\Omega_X, \calO_X) \cong H^0(X, \SExt^1(\Omega_X, \calO_X))$, and 
            \item there is an injection of obstruction spaces
            \[  \Ext^2(\Omega_X, \calO_X) \hookrightarrow \oplus_{i=0}^3 H^0(U_i, \SExt^2(\Omega_{U_i}, \calO_{U_i})) .\]
        \end{enumerate}
    In particular, if a deformation of each $U_i$ exists and they agree on the overlaps $U_i \cap U_j$, these glue together to give a deformation of $X$.
\end{theorem}

Using the previous theorem together with other techniques, we provide new examples of weighted projective spaces in every dimension $\ge 4$ that are smoothable to $\bP^n$ and that are not encompassed by Theorem \ref{thm:infiniteFamiliesDegenerationsPn}.

\begin{theorem}
    The two
    weighted projective threefolds $\bP(2,3,3,4)$ and $\bP(2,12,21,49)$
    are each smoothable to a cubic threefold. 
    Moreover,
    the weighted projective spaces $\bP(2,3,3,4,6_{n-4},18)$ and $\bP(2,12,21,42_{n-4},49,126)$ are smoothable to $\bP^{n}$ for any $n \ge 4$.
\end{theorem}

In the last section of the paper, we provide several applications of the construction of these degenerations of $\bP^n$. 

\subsection{Applications to terminal and canonical degenerations of $\bP^n$}
Using both orbifold cones (with isolated singularities) and the weighted projective spaces that admit smoothings to $\bP^n$ (which are singular in codimension 2), we produce new terminal and canonical degenerations of $\bP^n$ in every dimension $n \ge 4$.

First, we study degenerations $X$ with isolated singularities (constructed as orbifold cones).  These are of particular interest because, given a very ample linear system $L$ on $X$, a general section of $|L|$ will be a smooth divisor on $X$.  Choosing $L$ to be a multiple of the anticanonical linear system on $X$, these divisors deform to hypersurfaces in $\bP^n$, and thus one obtains smooth degenerations of hypersurfaces via this construction (cf. \cite{Mori75}).  

\begin{theorem}
    For any $n \ge 4$, there exist terminal degenerations of $\bP^n$ with isolated singularities.  If $n+1$ is composite and $n > 4$, there exist terminal Gorenstein degenerations of $\bP^n$ with isolated singularities. 
\end{theorem}

\begin{theorem}
    For any $n \ge 3$, there exist canonical degenerations of $\bP^n$ with isolated singularities.  For any odd number $n \ge 3$, there exist canonical, non-terminal Gorenstein degenerations of $\bP^n$ with isolated singularities.  
\end{theorem}

Even if $X$ does not have isolated singularities, if $X$ is canonical, a general element of the very ample linear system $L$ will also only have canonical singularities, so we can obtain canonical degenerations of hypersurfaces from canonical degenerations of $\bP^n$.  For infinitely many choices of $n$, the weighted projective spaces constructed in this paper that admit smoothings to $\bP^n$ are examples of such $X$ with canonical singularities.

\begin{theorem}
    Let $X$ be one of the weighted projective spaces 
        \begin{enumerate}
            \item $\bP(a^2,b^2,c^2)$, where $a^2 + b^2 + c^2 = 3abc$
            \item $\bP(a^2,b^2, 2c^2, 4abc)$, where $a^2 + b^2 + 2c^2 = 4abc$
            \item $\bP(2,3,3,4,18)$
            \item $\bP(2,12,21, 49, 126)$
        \end{enumerate}

    For $X$ as above, denote by $X_n$ be the associated weighted projective space that admits a smoothing to $\bP^n$, which are (respectively) 
        \begin{enumerate}
            \item $\bP(a^2,b^2,c^2,(abc)_{n-2})$, where $a^2 + b^2 + c^2 = 3abc$
            \item $\bP(a^2,b^2, 2c^2, 4abc, (2abc)_{n-3})$, where $a^2 + b^2 + 2c^2 = 4abc$
            \item $\bP(2,3,3,4,6_{n-4},18)$
            \item $\bP(2,12,21,42_{n-4}, 49, 126)$
        \end{enumerate}
        
    For each choice of $X$, there exist infinitely many $n > 0$ such that $X_n$ is a degeneration of $\bP^n$ with canonical Gorenstein singularities. 
\end{theorem}

\subsection{Applications to families of curves} 
Recall that canonical genus $g$ curves admit embeddings into $\bP^{g-1}$ and a smooth non-canonical genus $g$ curve is hyperelliptic.  One application of this work is that low degree hyperelliptic genus $g$ curves embed in $\bQ$-Gorenstein degenerations of $\bP^{g-1}$ in such a way that the hyperelliptic curves in the degeneration deform to canonical curves in $\bP^{g-1}$.

\begin{theorem}
    For $g = 3,4$, every smooth hyperelliptic genus $g$ curve embeds in a klt degeneration $X$ of $\bP^{g-1}$ as a (weighted) complete intersection.  The divisors defining the complete intersection deform in the smoothing of $X$ to $\bP^{g-1}$, and the resulting complete intersection on $\bP^{g-1}$ is a canonical genus $g$ curve. 
\end{theorem}

\subsection{Applications to moduli of canonically polarized varieties}
Next, we use degenerations of $\bP^4$ to produce \textit{smooth} degenerations of families of canonically polarized complete intersection surfaces that are themselves not complete intersections.  

\begin{theorem}
    There exists a smooth family of projective surfaces $\calS \to T$ over a pointed curve such that the generic fiber $\calS_t$ is a $(3,7)$ complete intersection in $\bP^4$ but the special fiber $\calS_0$ is a smooth surface that is not a complete intersection. 
\end{theorem}

It is conjectured that every smooth limit of a family of smooth hypersurfaces of prime degree $p$ in $\bP^n$ is again a hypersurface for $n \ge 4$ (see Section \ref{sec:FurtherApplications} for more details).  While we do not have a counterexample to this statement, the previous theorem provides a counterexample to the analogous result for complete intersections. 

We also use the weighted projective degenerations of $\bP^n$ to construct degenerations of degree $d$ hypersurfaces in $\bP^n$.  In particular, if $d$ is sufficiently divisible, as $n,d$ tend to infinity, the KSB moduli space $M^{hyp}_d$ generically parametrizing hypersurfaces of degree $d$ contains points corresponding to weighted hypersurfaces on arbitrarily many weighted projective spaces.

\begin{proposition}
    For any $d,n > 0$ such that $n \ge 2$ and $d > n+1$ and Markov triple $(a,b,c)$, if $abc \mid d$, there exist weighted hypersurfaces $V \subset \bP(a^2, b^2, c^2, (abc)_{n-2})$ of weighted degree $dabc$ parametrized by $M^{hyp}_d$.  Furthermore, the general such weighted hypersurface $V$ has only canonical singularities. 
\end{proposition}

\subsection{Applications to K-stability and moduli of Fano varieties}
As in the canonically polarized case, we can use the weighted projective degenerations of $\bP^n$ to construct K-stable degenerations of degree $d$ hypersurfaces in $\bP^n$, i.e. points parametrized by the boundary of the component of the K-moduli space $M^{hyp}_d$ generically parametrizing degree $d$ Fano hypersurfaces in $\bP^n$.  

\begin{proposition}
    For any $n \ge 2$ and Markov triple $(a,b,c)$ ordered so $a \le b \le c$, if $d \ge 2$ is such that $abc \mid d$, and $(1 - \frac{a}{bc})n + 1 < d <n+1$, there exist K-polystable Fano weighted hypersurfaces $V \subset \bP(a^2, b^2, c^2, (abc)_{n-2})$ of weighted degree $dabc$ parametrized by the K-moduli space $M^{hyp}_d$.  Furthermore, the general such weighted hypersurface $V$ has only canonical singularities. 
\end{proposition}

As a last application to K-stability, we prove that infinitely many of the weighted projective spaces in this paper are special degenerations of $\bP^n$. In what follows, denote by $F_n$ the Fibonacci sequence defined by $F_0 = 0$, $F_1 = 1$, and $F_{n}= F_{n-1}+F_{n-2}$.  

\begin{theorem} 
    The weighted projective spaces
    $\bP(1,F_{2n-1}^2,F_{2n+1}^2,(F_{2n-1}F_{2n+1})_k)$,
    where $n \ge 1$, $k \ge 0$,
    are special degenerations of $\bP^{k+2}$.
\end{theorem}

\begin{acknowledgement}
The authors thank Dori Bejleri, Andreas H\"oring, Junyao Peng, Andrea Petracci, and Chenyang Xu for helpful comments and discussions. K.D. was partially supported by NSF grant DMS-2302163. J.L. was supported by the Simons Foundation grant SFI-MPS-MOV-00006719-02. S.T. was supported by ANID, Fondecyt Postdoctoral Grant 3240013.
\end{acknowledgement}

\section{Preliminaries}

\subsection{Weighted projective space and cyclic quotient singularities.}
\label{subsec:WPSandCQS}
Throughout the paper, we work over $\mathbb{C}$.  We begin by summarizing definitions and conventions on weighted projective space for the convenience of the reader (more details can be found in \cite{Fletcher, Dolgachev, BeltramettiRobbiano}). A {\it weighted projective space} $\mathbb{P}(a_{0}, \dots, a_{n})$ {\it with weights} $a_{0}, \dots, a_{n} \in\mathbb{Z}_{>0}$ is the $n$-dimensional projective variety $\Proj\:\C[x_{0}, \dots, x_{n}]$, where the polynomial ring $\mathbb{C}[x_{0}, \dots, x_{n}]$ is graded by $\deg(x_{i}) = a_{i}$. 
We can equivalently define $\mathbb{P}(a_{0}, \dots, a_{n})$ as the quotient of $(\mathbb{A}^{n+1} - \{0\})/\mathbb{G}_{m}$, where $\mathbb{G}_{m} \cong \mathbb{C}^{\ast}$ is the one-dimensional algebraic torus acting by
\begin{center}
$(\lambda, (x_{0}, \dots, x_{n})) \mapsto (\lambda^{a_{0}}x_{0}, \dots, \lambda^{a_{n}}x_{n})$.
\end{center}

We can cover $\mathbb{P}(a_{0}, \dots, a_{n})$ by affine charts $U_i =(x_{i} \neq 0)$, where each chart $U_{i}$ is isomorphic to the quotient of $\mathbb{A}^{n}$ by the cyclic group $\mathbb{Z}_{a_i} =\mathbb{Z}/a_i\mathbb{Z}$,
with action given by the map
\begin{center}
$(\zeta_{a_i}, (z_{0}, \dots, \hat{z_i}, \dots, z_{n})) \mapsto (\zeta_{a_i}^{a_{0}}z_{0}, \dots, \hat{z_i}, \dots \zeta_{a_i}^{a_{n}}z_{n})$,
\end{center}
where $\zeta_{a_i}$ is a primitive $a_i$-th root of unity and $z_{j}=x_{j}/x_{i}^{a_j/a_{i}}$.

We say that a weighted projective space is {\it well-formed} if $\gcd(a_{0}, \dots, \hat{a_{i}}, \dots, a_{n}) = 1$ for all $0 \leq i \leq n$. Since any weighted projective space is isomorphic to a well-formed weighted projective space (see, e.g. \cite[Corollary 1.2.9]{Fletcher}), we will always assume that our weighted projective spaces are well-formed. 

Every weighted projective space other than $\bP^n=\mathbb{P}(1, \dots, 1)$ is singular, and the types of singularities that appear are cyclic quotient singularities, which in particular are klt. A description of the singular locus of $\mathbb{P}(a_{0}, \dots, a_{n})$ can be given in terms of the cyclic quotient singularities that appear.  

Recall that, given $m > 0$ and $b_i \in \bZ$, the {\it cyclic quotient singularity of type } ${\frac{1}{m}(b_{1}, \dots, b_{n})}$ is defined as the quotient $\mathbb{A}^{n}/\mathbb{Z}_{m}$ by the action $(\zeta_{m}, (z_{1}, \dots, z_{n})) \mapsto (\zeta_{m}^{a_{1}}z_{1}, \dots ,\zeta_{m}^{a_{n}}z_{n})$.  Writing $\bA^n = \Spec \ \bC[y_1, \dots, y_n]$, the quotient variety $\bA^n/\bZ_m$ is given by the spectrum of the ring of invariants, $\Spec \ \bC[y_1, \dots ,y_n]^{\bZ_m}$.

In a weighted projective space $\mathbb{P}(a_{0}, \dots, a_{n})$, the origin of each affine chart $U_i$ for $a_{i} > 1$ is a cyclic quotient singularity of type ${\frac{1}{a_{i}}(a_{0}, \dots,  \hat{a}_{i}, \dots, a_{n})}$. 
More generally, for any collection $I$ of weights $\{ a_i \}_{i \in I}$ such that $\gcd(\{a_i\}_{i \in I}) = d_I> 1$, there is a
singularity of type $\frac{1}{d_{I}}(\{a_j\}_{j \in \{0,1,\dots,n\}-I})$ at the generic point of the stratum $\{ x_j = 0 \mid j \in \{ 0,1,\dots,n\} - I\}$. This means that, near the generic point of this stratum, the singularity is of the form $\frac{1}{d_{I}}(\{a_j\}_{j \in \{0,1,\dots,n\}-I}) \times \bA^{|I|-1}$. 

The Weil and Cartier divisors on a weighted projective space are well understood.  The following can be found in \cite{BeltramettiRobbiano}:

\begin{theorem}{\cite[Theorem 4B.7, Theorem 7.1]{BeltramettiRobbiano}}\label{thm:divisorsonwps}
    Let $X = \bP(a_0, \dots, a_n)$ be a well-formed weighted projective space.  We denote by $\calO_X(d)$ the coherent sheaf associated to the module of sections of degree $d$ of $\bC[x_0, \dots, x_n]$, where $x_i$ has degree $a_i$. Then, 
        \begin{enumerate}
            \item For any $d \in \bZ$, $\calO_X(d)$ is reflexive and Cohen-Macaulay.
            \item $\Cl(X) \cong \bZ$, generated by $\calO_X(1)$, and 
            \item $\Pic(X) \cong \bZ$, generated by $\calO_X(m)$ where $m = \mathrm{lcm} \ (a_0, \dots, a_n)$.
        \end{enumerate}

    For the isomorphism in (2), the group law on $\Cl(X)$ is defined by the double dual of the tensor product $(\calO(a) \otimes \calO(b))^{**} \cong \calO(a+b)$.  It need not hold that \[ \calO(a) \otimes \calO(b) \to \calO(a+b) \] is an isomorphism, although this is true if either $\calO(a)$ or $\calO(b)$ is Cartier. 
\end{theorem}

Intersections of divisors on weighted projective space are also well understood. 

\begin{theorem}\label{thm:intersections-on-wps}
    Let $D_1, \dots, D_n$ be divisors on $X = \bP(a_0, \dots, a_n)$, a well-formed $n$-dimensional weighted projective space, and suppose that $\deg(D_{i}) = d_{i}$ for $i = 1, \dots, n$.  Then, 
    \[ D_1 \cdot \ldots \cdot D_n = \frac{d_1 \dots d_n}{a_0 \dots a_n}.\]
    In particular, if $D$ is a divisor on $X$ and $\deg(D) = d$, then
    \[ D^n = \frac{d^n}{a_0 \dots a_n}.\]
\end{theorem}

\begin{proof}
    Let $\pi: \bP^n \to X$ be the canonical projection map induced by the inclusion of rings $\bC[y_0^{a_0}, \dots, y_n^{a_n}] \hookrightarrow \bC[y_0, \dots, y_n]$, where $\deg y_i = 1$ 
    and $\bP^n = \Proj\:\C[y_0, \dots, y_n]$. 
    Then $\deg(\pi) = a_0 \dots a_n$ (cf. \cite[Theorem 3A.1]{BeltramettiRobbiano}) and $\pi^*(D_i)$ has degree $d_i$ on $\bP^n$, so by the usual intersection theory on $\bP^n$,
    \[ \pi^*(D_1) \cdot \ldots \cdot \pi^*(D_n) = d_1 \dots d_n. \]
    By the projection formula, 
    \[\pi^*(D_1) \cdot \ldots \cdot \pi^*(D_n) = \deg(\pi) D_1 \cdot \ldots \cdot D_n \] so we conclude
    \[ D_1 \cdot \ldots \cdot D_n = \frac{d_1 \dots d_n}{a_0 \dots a_n}.\]
\end{proof}

We next include several results on the canonical divisor and sheaf of differentials on weighted projective spaces. 

\begin{theorem}\label{thm:wps-dualizingprops}
    Let $X = \bP(a_0, \dots, a_n)$ be a well-formed weighted projective space.  Then, 
        \begin{enumerate}
            \item The dualizing sheaf and canonical divisor are given by $\omega_X \cong \calO(K_X) = \calO(-a_0 - \dots - a_n)$.  
            \item Let $\Omega_X^{[1]}$ denote the double dual of the sheaf of K\"{a}hler differentials.  If $U \subset X$ denotes the smooth locus and $i: U \to X$ the inclusion, $\Omega_X^{[1]} = i_* \Omega_U$. 
            \item There is an exact sequence, analogous to the Euler sequence on projective space,
            \[ 0 \to \Omega_X^{[1]} \to \oplus_{i = 0}^n \calO_X(-a_i) \to \calO_X \to 0.\]
            \item $\Omega_X^{[1]}$ is Cohen-Macaulay.
            \item There is a natural map $\Omega_X^1 \to \Omega_X^{[1]}$, but this need not be injective nor surjective.
        \end{enumerate}
\end{theorem}

\begin{proof}
    Part (1) is \cite[Corollary 6B.8]{BeltramettiRobbiano}.  Part (2) is \cite[Theorem 6A.12]{BeltramettiRobbiano} combined with \cite[Theorem 3]{Kni73}, realizing the double dual as the sheaf of invariants of the pushforward $\pi: \bP^n \to X$ and \cite[Lemma 6.A11]{BeltramettiRobbiano} identifying this sheaf of invariants with $i_* \Omega_U$.  Part (3) follows from the description of the differentials and the Euler formula in \cite[2.1.2, 2.1.3, following Lemma]{Dolgachev}.  The statement that $\Omega_X^{[1]}$ is Cohen-Macaulay in (4) follows because it is the kernel of a map between two Cohen-Macaulay modules.  Finally, the map in (5) $\Omega_X^1 \to \Omega_X^{[1]}$ is precisely the double dual morphism, and examples in which this is not injective or surjective are found in \cite[Table 1]{GR11}.
\end{proof}

We conclude this section with a description of the cohomology of several sheaves on weighted projective space. 

\begin{theorem}\label{thm:cohomologyofwps}
    Let $X = \bP(a_0, a_1, \dots, a_n)$ be a well-formed weighted projective space.  Then, for $d > -a_0 - \dots -a_n$ and $0< i \le n$, 
    \[ H^i(X, \calO_X(d)) = 0.\]
   For $1 < i \le n$, 
    \[ H^i(X, \Omega_X^{[1]}) = 0.\]
    For $1 \le i \le n$ and $T_X = (\Omega_X^{1})^* = (\Omega_X^{[1]})^*$,
    \[ H^i(X, T_X) = 0.\]
\end{theorem}

\begin{proof}
    This is well-known.  For the first statement, see \cite[1.4.2]{Dolgachev}. 
    For the second, use the exact sequence from Theorem \ref{thm:wps-dualizingprops} (3) and the first statement.  For the third, dualize the exact sequence in Theorem \ref{thm:wps-dualizingprops} (3) to obtain 
    \[ 0 \to \calO_X \to \oplus_{i=0}^n \calO_X(a_i) \to T_X \to 0 ,\]
    noting that this is exact on the right because $\SExt^1(\calO_X, \calO_X)= 0$. Applying the first statement yields the desired result. 
\end{proof}

\subsection{$\mathbb{Q}$-Gorenstein smoothings and a necessary condition for smoothability of weighted projective space}

In this section, we define the notions of $\mathbb{Q}$-Gorenstein deformation and $\bQ$-Gorenstein smoothing.  We work only over a reduced base scheme as that is sufficient for the purpose of this paper. 

\begin{definition}
\label{def:Q-GorensteinFamily}
A flat family of schemes $\mathcal{X} \rightarrow B$ over a reduced base $B$ is said to be {\it $\mathbb{Q}$-Gorenstein} if the relative canonical divisor $K_{\mathcal{X}/B}$ is $\mathbb{Q}$-Cartier.
\end{definition}

\begin{definition}
    A deformation of a variety $X$ is a flat family
    $\calX \to B$ 
    over a pointed reduced scheme $B$ with closed point $0 \in B$ such that $X \cong \calX_0$.
    For general $b \in B$,
    we will say that $\calX_b$
    is a \textbf{deformation} of $\calX_0$ and
    $\calX_0$ is a \textbf{degeneration} of $\calX_b$,
    and will use the notation $\calX_b \rightsquigarrow \calX_0$.
    If $\calX_b$ is smooth, we will also say $\calX_b$ is a \textbf{smoothing} of $\calX_0$.
    We will use the terminology ``$\bQ$-Gorenstein smoothing'' or ``$\bQ$-Gorenstein degeneration'' 
    whenever the family $\calX \to B$ is $\bQ$-Gorenstein.
\end{definition}

While $B$ can be an arbitrary reduced scheme in the previous definition, we will primarily consider deformations over smooth curves. 

The $\bQ$-Gorenstein condition provides a necessary condition for a weighted projective space to admit a $\bQ$-Gorenstein smoothing to $\bP^n$.  The following proposition is originally stated as Proposition 4.48 in \cite{DeVleming}; we include a proof below for the convenience of the reader.

\begin{proposition}
\label{prop:necessaryConditionSmoothing}
Suppose that $\mathbb{P}(a_{0}, \dots, a_{n})$ admits a $\mathbb{Q}$-Gorenstein smoothing to $\mathbb{P}^{n}$. Then the following equation must hold:
\begin{equation}
\label{Eqn:neccessCond-QGsmoothingPn}
(n+1)^{n}a_{0}a_{1} \dots a_{n} = (a_{0} + a_{1} + \dots + a_{n})^{n}.
\end{equation}
\end{proposition}

\begin{proof}  
Suppose we have a well-formed weighted projective space $\mathbb{P}(a_{0}, \dots, a_{n})$ that admits a $\mathbb{Q}$-Gorenstein smoothing to $\mathbb{P}^{n}$. 
By \cite[Theorem-Definition 3.1]{Kollar23} the anticanonical volume is constant in $\mathbb{Q}$-Gorenstein families, so we must have
\begin{equation*}
(-K_{\mathbb{P}^{n}})^{n} = (-K_{\mathbb{P}(a_{0}, \dots, a_{n})})^{n}
\end{equation*}

The canonical divisor of $\mathbb{P}(a_{0}, \dots, a_{n})$ is given by $\calO(K_{\bP(a_0, \dots, a_n)}) = \calO(-a_0 - \dots - a_n)$, and by the intersection theory in Theorem \ref{thm:intersections-on-wps}, the equality $(-K_{\mathbb{P}^{n}})^{n} = (-K_{\mathbb{P}(a_{0}, \dots, a_{n})})^{n}$ is equivalent to 
\[ (n+1)^n = \frac{(a_0 + \dots + a_n)^n}{a_0 \dots a_n}, \]
which in turn is equivalent to Equation \eqref{Eqn:neccessCond-QGsmoothingPn}.
\end{proof}

Miraculously, when $n = 2$, Equation \eqref{Eqn:neccessCond-QGsmoothingPn} is both a necessary and sufficient condition for a weighted projective space to admit a $\bQ$-Gorenstein smoothing to $\bP^2$.  This was proved by Hacking and Prokhorov in \cite{HP10}. Since the same ideas will be useful for the purpose of the present article,
we include a proof below. We first provide an easy preliminary lemma. 

\begin{lemma}\label{lem:P2andmarkoveqn}
    A triple $(a_0,a_1,a_2)$ of pairwise coprime positive integers satisfy the equation \[9a_0a_1a_2 = (a_0+a_1+a_2)^2\] if and only if $a_0 = a^2$, $a_1 = b^2$, and $a_2 = c^2$ where $a,b,c$ are positive integers satisfying \[a^2 + b^2 + c^2 = 3abc.\] 
\end{lemma}

Before we prove the lemma, let us recall that triples $(a,b,c)$ satisfying $a^2 + b^2 + c^2 = 3abc$ are known as \textit{Markov triples}, and $(1,1,1)$ can be obtained from any Markov triple by a finite sequence of \textit{mutations}, replacements $(a,b,c) \mapsto (a,b,3ab-c)$ or $(a,b,c) \mapsto (a,3ac-b,c)$ or $(a,b,c) \mapsto (3bc-a,b,c)$ (cf. \cite{KN98}). 

\begin{proof}
Suppose first that $(a_0,a_1,a_2)$ is a triple of pairwise coprime positive integers satisfying
    \begin{align*}
9a_{0}a_{1}a_{2} &= (a_{0} + a_{1} + a_{2})^{2},  \text{ which implies }\\
a_{0}a_{1}a_{2} &= \displaystyle{\frac{(a_{0} + a_{1} + a_{2})^{2}}{9}} \in \mathbb{Z}, \text{ and hence} \\
\sqrt{a_{0}a_{1}a_{2}} &= \displaystyle{\frac{a_{0} + a_{1} + a_{2}}{3}} \in \mathbb{\Q}.
\end{align*}
By the rational root theorem, as the polynomial $x^2  - \frac{a_{0} + a_{1} + a_{2}}{3} $ is monic, any rational root must in fact be an integer, and hence 
\[ \sqrt{a_{0}a_{1}a_{2}} = \displaystyle{\frac{a_{0} + a_{1} + a_{2}}{3}} \in \mathbb{\Z}. \]
As the $a_i$'s are pairwise coprime, it follows that each $a_{i}$ for $i = 0, 1, 2$ is a perfect square. Writing
$a_{0} = a^{2}$, $a_{1} = b^{2}$ and $a_{2} = c^{2}$
and plugging in to the equation $3\sqrt{a_0a_1a_2} = a_0+a_1+a_2$ proves that $3abc = a^{2} + b^{2} + c^{2}$.

Now, suppose that $(a,b,c)$ is a Markov triple.  Necessarily, we claim that the numbers $a,b,c$ must be pairwise coprime.  First, observe that if any two have a common prime factor, the equality $3abc = a^{2} + b^{2} + c^{2}$ implies that the third must also be divisible by that factor.  As any Markov triple can be decreased by a finite sequence of mutations to $(1,1,1)$, and these mutations preserve any common prime factors, we see that $(a,b,c)$ cannot have a common prime factor.  Therefore, they must be pairwise coprime.  This implies that $a_0 = a^2, a_1 = b^2, a_2 = c^2$ are pairwise coprime and squaring both sides of the equation $3abc = a^{2} + b^{2} + c^{2}$ yields $9a_0a_1a_2 = (a_0 + a_1 + a_2)^2$.
\end{proof}

Now, we rephrase a result of Hacking and Prokhorov using the previous lemma.

\begin{theorem}\cite{HP10}
\label{thm:wpdegensofP2}
A well-formed weighted projective surface $\bP(a_0, a_1, a_2)$ admits a $\bQ$-Gorenstein smoothing to $\bP^2$ if and only if $9a_0a_1a_2 = (a_0+a_1+a_2)^2$.  Equivalently, by Lemma \ref{lem:P2andmarkoveqn}, this holds if and only if $a_0 = a^2$, $a_1 = b^2$, and $a_2 = c^2$ where $3abc = a^2 + b^2 + c^2$.
\end{theorem}

\begin{proof}
Suppose we have a well-formed weighted projective surface $X = \mathbb{P}(a_{0}, a_{1}, a_{2})$ that admits a $\mathbb{Q}$-Gorenstein smoothing to $\mathbb{P}^{2}$.
We know that $9a_0a_1a_2 = (a_0+a_1+a_2)^2$ is a necessary condition from Equation \eqref{Eqn:neccessCond-QGsmoothingPn}, and by Lemma \ref{lem:P2andmarkoveqn}, this is equivalent to $a_0 = a^2$, $a_1 = b^2$, and $a_2 = c^2$ where $3abc = a^2 + b^2 + c^2$.  Write $\bP(a_0,a_1,a_2) = \bP(a^2,b^2,c^2)$.  This demonstrates the `only if' condition.

Remarkably, all such weighted projective surfaces are $\bQ$-Gorenstein smoothable.  For the proof of the `if' direction, we first show the local smoothability of each singularity.\footnote{There are several other methods to do this: one could assert the existence of a $\bQ$-Gorenstein smoothing by writing these singularities as $T$-singularities (cf. \cite{HP10}) or referring to the full computation of deformations of cyclic quotient singularities in \cite[Section 6.6]{Kollar23}.}

Let $c' = 3ab - c$, so that $a^2 + b^2 = cc'$, and
let $\bP(a^2, b^2, c^2) \hookrightarrow \bP(a^2, b^2, c',c)$
be the {\it degree $c$ embedding} given by $[x:y:z] \mapsto [x^c:y^c:xy:z]$.  Denoting the new coordinates by $[x_0:x_1: x_2:x_3]$, the image is defined by $x_0x_1 = x_2^c$.  Now, consider the family over $\bA^1_t$ given by $x_0x_1 = x_2^c + tx_3^{c'}$.  When $t = 0$, this is isomorphic to $\bP(a^2, b^2,c^2)$, but when $t \ne 0$, the corresponding surface only has two singular points of type $\frac{1}{a^2}(b^2,c^2)$ and $\frac{1}{b^2}(a^2,c^2)$.  It no longer has a singularity of type $\frac{1}{c^2}(a^2,b^2)$, and hence this gives an explicit smoothing of the singularity.  The total space of this family is a hypersurface in the $\bQ$-Gorenstein variety $\bP(a^2,b^2,c',c) \times \bA^1$, so it is $\bQ$-Gorenstein and this is in fact a $\bQ$-Gorenstein smoothing. 

Up to permuting the variables, applying the same construction gives a local smoothing of each singularity on $\bP(a^2,b^2,c^2)$, so each singular point is smoothable.  Finally, because $X = \bP(a^2, b^2,c^2)$ has isolated singularities and $H^2(X,T_X) = 0$ from Theorem \ref{thm:cohomologyofwps}, by \cite[Proposition 2.3]{PetracciToric}, there are no local-to-global obstructions to extending these local smoothings to a global smoothing of $X$, and hence $X$ is $\bQ$-Gorenstein smoothable.  Because $-K_X$ is ample and ampleness is an open condition in families, any such smoothing of $X$ must be a smooth Fano surface with anticanonical volume $9$.  The only such surface is $\bP^2$, so we conclude any weighted projective surface $\bP(a^2, b^2,c^2) $ such that $a^2+b^2+c^2 = 3abc$ is smoothable to $\bP^2$.
\end{proof}

In Example \ref{ex:solution-n2-to-n3}, we demonstrate how, given a $\mathbb{Q}$-Gorenstein smoothing of $\mathbb{P}(a_{0}, a_{1}, a_{2})$ to $\mathbb{P}^{2}$, we may use Equation \eqref{Eqn:neccessCond-QGsmoothingPn} to find a weight $a_{3}$ such that $(a_{0}, a_{1}, a_{2}, a_{3})$ again satisfies Equation \eqref{Eqn:neccessCond-QGsmoothingPn} for $n = 3$. 
It is not immediate that
the new weighted projective space $\mathbb{P}(a_{0}, a_{1}, a_{2}, a_{3})$
admits a smoothing to $\mathbb{P}^{3}$, 
but as a consequence of Theorem \ref{thm:cone}, we will show in later sections that it actually does.

\begin{example}
\label{ex:solution-n2-to-n3}
Suppose we have a $\mathbb{Q}$-Gorenstein smoothing of $\mathbb{P}(a_{0}, a_{1}, a_{2})$ to $\mathbb{P}^{2}$ which by Theorem \ref{thm:wpdegensofP2} is equivalent to
$a_{0} = a^{2}$, $a_{1} = b^{2}$ and $a_{2} = c^{2}$,
such that $3abc = a^{2} + b^{2} + c^{2}$. We may find $a_{3} \in \mathbb{Z}$ such that $a_{0}, \dots, a_{3}$ satisfy Equation \eqref{Eqn:neccessCond-QGsmoothingPn} with $n = 3$. That is, we want
\begin{align*}
(3+1)^{3} a_{0}a_{1}a_{2}a_{3} &= (a_{0} + a_{1} + a_{2} + a_{3})^{3}, \text{ or} \\
64a^{2}b^{2}c^{2}a_{3} &= (3abc + a_{3})^{3},
\end{align*}
and taking $a_{3} = abc$ makes the above equation true, and thus we have found a solution for $n = 3$.  We will show in Section \ref{sec:cone} that the weighted projective space $\bP(a^2,b^2,c^2,abc)$ in fact admits a smoothing to $\bP^3$.
\end{example}

\noindent Producing solutions as shown in \Cref{ex:solution-n2-to-n3} can be generalized as the following result (originally Proposition 4.52 in \cite{DeVleming}). 
Note that the last weight $a_n$ is in fact an integer, which follows from the rational root theorem.
This will be expanded later in \Cref{prop:relationsonweights}.

\begin{proposition}
\label{Prop:SolnTon+1}
Given any solution to the equation
\begin{equation}
\label{Eqn:solution:n}
n^{n-1} \prod_{i=0}^{n-1} a_{i} = \Biggl(\sum_{i=0}^{n-1} a_{i} \Biggr)^{n-1},
\end{equation}
we may use this solution to determine a solution for the next case, satisfying the equation
\begin{equation}
\label{Eqn:solution:n+1}
(n+1)^{n} \prod_{i=0}^{n} a_{i} = \Biggl(\sum_{i=0}^{n} a_{i} \Biggr)^{n},
\end{equation}
where $a_{0}, \dots, a_{n-1}$ are the same as in the first equation, and we fill in
\begin{align*}
a_{n} = \displaystyle{\frac{\Bigl( \; \sum\limits_{i=0}^{n-1} a_{i} \Bigr)}{n}} = \Biggl( \; \prod_{i=0}^{n-1} a_{i} \Biggr)^{1/(n-1)}\in\bZ.
\end{align*}
\end{proposition}

\begin{remark}\label{rmk:avg}
    In \Cref{Prop:SolnTon+1}, we observe that the choice for weight $a_{n}$ is the average of the previous $n-1$ weights. Conversely, it is straightforward to show that if $\bP(a_0, \dots, a_n)$ has the same anticanonical volume as $\bP^n$ and one weight $a_n$ is the average of the remaining $n$ weights, then the weighted projective space $\bP(a_0, \dots, a_{n-1})$ has the same anticanonical volume as $\bP^{n-1}$.
\end{remark}

\begin{remark}
If the weighted projective space $\bP(a_0, \dots, a_{n-1})$ from Equation \eqref{Eqn:solution:n+1} does smooth to $\mathbb{P}^{n-1}$, then geometrically, adding a final weight $a_n$ corresponds to taking a `cone' over the family $\mathbb{P}^{n-1}$ degenerating to $\mathbb{P}(a_{0}, \dots, a_{n-1})$. As we will show, this gives a family of $\mathbb{P}^{n}$ degenerating to $\mathbb{P}(a_{0}, \dots, a_{n-1}, a_{n})$. This will be explored in more detail in \Cref{sec:cone}.
\end{remark}

\subsection{Schlessinger's rigidity theorem}

We have seen in \Cref{thm:wpdegensofP2} that the solutions $(a_0, a_1, a_2)$ to the equation \[9a_0a_1a_2 = (a_0+a_1+a_2)^2\] correspond exactly to weighted projective surfaces admitting a $\bQ$-Gorenstein smoothing to $\bP^2$.  One asks if this behavior persists in higher dimensions.  By Proposition \ref{prop:necessaryConditionSmoothing}, it is necessary that \[(n+1)^n(a_0 \dots a_n) = (a_0 + \dots + a_n)^n\] for a weighted projective $\bP(a_0, \dots, a_n)$ to admit a $\bQ$-Gorenstein smoothing to $\bP^n$, but is this sufficient?  In higher dimensions, the answer is no: the following theorem of Schlessinger (see Theorem 3 in \cite{S71}) allows us to conclude that in certain cases a smoothing {\it does not} exist.  However, because $\bP^n$ is the only smooth Fano variety with anticanonical volume $(n+1)^n$, producing a smoothing of a weighted projective space with this volume will be sufficient to conclude it admits a smoothing to $\bP^n$. 

\begin{theorem}
[Schlessinger]
\label{thm:Schlessinger}
Let $Y$ be smooth and suppose that $\dim Y \geq 3$. Let $G$ be a finite group acting on $Y$ with only one fixed point. Then the quotient $X = Y/G$ is rigid.
\end{theorem}

There is a variant for toric varieties, which comes from the following result of Totaro (see Theorem 5.1 in \cite{T12}):

\begin{theorem}
[Totaro]
\label{thm:rigidToricFano} A toric Fano variety which is smooth in codimension 2 and $\mathbb{Q}$-factorial in codimension 3 is rigid.
\end{theorem}

When we apply \ref{thm:rigidToricFano} to the weighted projective space setting, we obtain the following corollary.

\begin{corollary}
\label{thm:GeneralizedSchlessinger}
A weighted projective space that is smooth in codimension 2 
is rigid.
\end{corollary}

A local version of Schlessinger's theorem for weighted projective space is given below.

\begin{corollary}\label{cor:codim3notsmoothable}
    A weighted projective space $\bP(a_0, \dots, a_n)$ where there is a subset of $k< n-1$ weights with a common factor but no set of $n-1$ weights sharing a divisor of that common factor is not smoothable.
\end{corollary}

\begin{proof}
    Suppose $\gcd(a_0, \dots, a_{k-1}) = d > 1$ for some set of $k< n-1$ weights of $\bP(a_0, \dots, a_n)$,
    where we may assume the rest of the weights are coprime with $d$.
    Then, the generic point of the stratum $[x_0: \dots :x_{k-1}: 0: \dots :0 ] = (x_{k} = 0 ) \cap \dots \cap (x_n = 0)$ is an isolated quotient singularity of type $\frac{1}{d}(a_{k}, \dots, a_n)$.  Because $k < n-1$, this singularity has codimension at least 3, and therefore is not smoothable. 
\end{proof}

In the next example we show, using Theorem \ref{thm:Schlessinger}, that in higher dimensions the Equation \eqref{Eqn:neccessCond-QGsmoothingPn} is not sufficient for the existence of a smoothing.

\begin{example}
\label{ex:Schlessinger1}
Consider the $3$-dimensional weighted projective space $Y = \mathbb{P}(1,4,16,27)$. The weights satisfy Equation \eqref{Eqn:neccessCond-QGsmoothingPn}, yet $Y$ is not smoothable by Corollary \ref{cor:codim3notsmoothable}: the last weight $a_{3} = 27$ is relatively prime to the three other weights, so there is a cyclic quotient singularity of type ${\frac{1}{27}(1, 4, 16)} = \mathbb{A}^{3}/\mathbb{Z}_{27}$ at the point $[0:0:0:1]$. That is, $Y$ has an isolated cyclic quotient singularity in codimension 3. It follows by Theorem \ref{thm:Schlessinger} or Corollary \ref{cor:codim3notsmoothable} that this singular point is rigid and in particular $Y$ has no smoothing to $\mathbb{P}^{3}$.
\end{example}

\begin{example}
\label{ex:Schlessinger2}
Consider $Y = \mathbb{P}(1,4,12_k,16, 27)$ for any $k \ge 0$. These satisfy Equation \eqref{Eqn:neccessCond-QGsmoothingPn} by Remark \ref{rmk:avg}. Since $\gcd (12, \dots, 12, 27) = 3$, the generic point of the corresponding index 3 cyclic quotient singularity is a point of codimension 3. By Corollary \ref{cor:codim3notsmoothable}, there is no smoothing of $Y$ to $\mathbb{P}^{k+3}$.
\end{example}

While Equation \eqref{Eqn:neccessCond-QGsmoothingPn} is not sufficient to produce a smoothing of a weighted projective space, it \textit{is} sufficient to prove that the space is locally smoothable in codimension 2. This can be used to give an alternative proof of Theorem \ref{thm:wpdegensofP2}.

To prove this, we begin with observations on the solutions to the volume equation.

\begin{proposition}\label{prop:relationsonweights}
    Suppose $\bP(a_0, \dots, a_n)$ is a well-formed weighted projective space with weights satisfying the equation 
    \[ (n+1)^{n} \prod_{i=0}^{n} a_{i} = \Biggl(\sum_{i=0}^{n} a_{i} \Biggr)^{n}.\]  Then, 
        \begin{enumerate}
            \item The $n$th root of the product of the weights is equal to the average of the weights, \[ (a_{0} \cdot \dots \cdot a_{n})^{1/n} = \displaystyle{\frac{a_{0} + a_{1} + a_{2} + \cdots + a_{n}}{n+1}},\] and both are integers. 
            \item Suppose $n-1$ weights have a nontrivial common factor $d$ and let $d = d_1^{p_1} \dots d_k^{p_k}$ be the prime factorization of $d$.  Without loss of generality, denote these weights by $a_2, \dots, a_n$.  Then, $a_0 + a_1$ is divisible by $r = d_1^{q_1} \dots d_k^{q_k}$ where $p_i \le 2q_i$ for each $i$.
        \end{enumerate}
\end{proposition}

\begin{proof}
From 
\begin{equation*}
(n+1)^{n} \prod_{i=0}^{n} a_{i} = \Biggl(\sum_{i=0}^{n} a_{i} \Biggr)^{n},
\end{equation*}
taking the $n$-th root of both sides, we obtain
\begin{equation*}
(a_{0} \cdot \dots \cdot a_{n})^{1/n} = \displaystyle{\frac{a_{0} + a_{1} + a_{2} + \cdots + a_{n}}{n+1}},
\end{equation*}
which is a rational number. So $m := (a_{0} \cdot \dots \cdot a_{n})^{1/n} \in \mathbb{Q}$ is a root of $f(x) = x^{n} - a_{0} \cdots a_{n}$. Because the $n$-th root is rational, it follows from the rational root test that $m$ is an integer.

Now suppose that the weights $a_{2}, \dots, a_{n}$ in $\mathbb{P}(a_{0}, a_{1}, a_{2}, \dots, a_{n})$ have greatest common divisor $d \in \mathbb{Z}$ and denote the prime factorization of $d$ by $d_{1}^{p_{1}} d_{2}^{p_{2}} \dots d_{k}^{p_{k}}$. Observe that
\begin{enumerate}
\item None of the prime factors of $d$ divide $a_{0}$ nor $a_{1}$; this comes from the assumption that $\mathbb{P}(a_{0}, \dots, a_{n})$ is well-formed.
\item Each prime factor $d_{i}$ of $d$ must divide $a_{0} \dots a_{n}$, so each $d_{i}$ must divide $m$. 
In fact, for each prime factor $d_i$ of $d = d_1^{p_1} \dots d_k^{p_k}$, then $d_i^{p_i}$ divides $a_2, \dots, a_n$, so $d_i^{(n-1)p_i}$ divides $a_2 \dots a_n$.
Let $j_i$ be the largest power of $d_i$ that divides $a_2 \dots a_n$, so $j_i \ge (n-1)p_i$. 
As $m = (a_0 \dots a_n)^{1/n} \in \bZ$, this implies $j_i$ is divisible by $n$,
so $j_i = n s_i$ for some $s_i > 0$ and $d_i^{s_i}$ divides $m$.  From the equation 
\[ (n+1) m = a_0 + a_1 + \dots + a_n, \]
we conclude $d_i^{s_i}$ divides $a_0 + \dots + a_n$.
Let $r = d_1^{q_1} \dots d_k^{q_k}$ where $q_i = \min \{ p_i, s_i \}$. Because $d_i^{s_i}$ for each $i$ and hence $d_1^{s_1} \dots d_k^{s_k}$ divides $a_0 + \dots + a_n$, and $d = d_1^{p_1} \dots d_k^{p_k}$ divides $a_2 + \dots + a_n$, we conclude $r$ divides $a_0 + a_1$.

\end{enumerate}

To finish the proof, we must show $p_i \le 2q_i$, where $ d = d_1^{p_1} \dots d_k^{p_k}$ and $r = d_1^{q_1} \dots d_k^{q_k}$ and $q_i = \min \{ p_i, s_i \}$ where $s_i = \frac{j_i}{n} \ge \frac{n-1}{n}p_i$ is as above.  Certainly $p_i \le 2 p_i$, while $p_i \le 2 \frac{(n-1)}{n}p_i \le 2 s_i$. 
Therefore, $p_i \le 2 \min \{p_i, s_i\}=2 q_i$ and the proof is complete.  
\end{proof}

\begin{corollary}
\label{cor:T-singularities}
    Suppose $\bP(a_0, \dots, a_n)$ is a well-formed weighted projective space with weights satisfying the equation 
    \[ (n+1)^{n} \prod_{i=0}^{n} a_{i} = \Biggl(\sum_{i=0}^{n} a_{i} \Biggr)^{n}.\]  Then, $\bP(a_0, \dots, a_n)$ has only $T$-singularities in codimension $2$.  If any set of $n-1$ weights has a non-trivial highest common factor $d$ with $d = \mathrm{rad}(d)$ (e.g. if $d$ is prime), then $\bP(a_0, \dots, a_n)$ has only $A_{d-1}$ singularities along the corresponding singular stratum.
\end{corollary}

\begin{proof}
    The singularities in codimension 2 correspond exactly to subsets of $n-1$ weights with a nontrivial common factor.  Without loss of generality, assume these weights are $a_2, \dots, a_n$, and $d = \gcd(a_2, \dots, a_n)$. The associated singularity in codimension 2 of the weighted projective space is a $\frac{1}{d}(a_0, a_1)$ singularity.  The well-formed assumption implies that $\gcd(a_0, d) = \gcd(a_1, d) = 1$.  
    In particular, we may replace the primitive $d$th root of unity $\zeta_d$ by the primitive root $(\zeta_d)^{a_0^{-1}}$, where $a_0^{-1} \in \bZ_d$ is the multiplicative inverse of $a_0$, to write this $\frac{1}{d}(a_{0},a_{1})$ singularity equivalently as $\frac{1}{d}(1, b)$, where $b \equiv a_0^{-1} a_1 \mod d$.

    Write the prime factorization of $d$ as $d = d_{1}^{p_{1}} d_{2}^{p_{2}} \dots d_{k}^{p_{k}}$. From \Cref{prop:relationsonweights}, we see that $a_0 + a_1$ is divisible by $r = d_1^{q_1} \dots d_k^{q_k}$, where $p_i \le 2q_i$ for each $i$.  
    Here we may assume that each $d_i^{q_i}$ is the highest power of $d_i$ dividing $a_0+a_1$.
  
    Because $a_0 + a_1 \equiv 0 \mod r$ 
    and $a_0, a_1$ are relatively prime to $d$, we also have $1 + b \equiv 0 \mod r$. Therefore, $b = ar - 1$ for some $a > 0$ 
    with $\gcd(a,r) = 1$ by assumption on $r$.
    Write $d = d_1^{\; p_1 - q_1} \dots d_k^{\; p_k-q_k} r$
    and recall that
    $p_i \le 2q_i$ by the previous proposition.
    Letting $n = d_1^{p_1 - q_1} \dots d_k^{p_k-q_k}$ and $e = d_1^{\; 2q_1 - p_1} \dots d_k^{\; 2q_k - p_k}$, we may write \[\frac{1}{d}(1,b) = \frac{1}{d}(1,ar-1) =  \frac{1}{en^2}(1,ena - 1).\] 
    to realize this as a $T$-singularity. 

    Moreover, if $d = \mathrm{rad}(d)$, i.e. $d = d_1 \dots d_k$, then $d = r$. Hence $a_0 + a_1$ is divisible by $d$ and so is $1 + b$. That is, $b \equiv -1 \mod d$, so this is a $\frac{1}{d}(1,-1) = A_{d-1}$ singularity.
\end{proof}

As $T$-singularities are smoothable, we have the following:

\begin{corollary}
    Let $\bP(a_0, \dots, a_n)$ be a weighted projective space satisfying Equation \eqref{Eqn:neccessCond-QGsmoothingPn}.  Then, every codimension 2 point is 
    $\bQ$-Gorenstein
    smoothable. 
\end{corollary}

\begin{remark}
    This corollary seems 
    to be 
    related to recent work of Koll\'ar and Kov\'acs on the $\bQ$-Gorenstein condition in codimension 2.  In \cite{KK23}, the authors prove that the $\bQ$-Gorenstein condition is automatic in codimension $3$, i.e. given a family of varieties $\calX \to T$ over a smooth base $T$ such that $K_{\calX/T}$ is $\bQ$-Cartier at all points of codimension 2 in each fiber, then $K_{\calX/T}$ is $\bQ$-Cartier at all points.  
\end{remark}

We now have two necessary conditions for a well-formed weighted projective space $X = \bP(a_0, \dots, a_n)$ to admit a $\bQ$-Gorenstein smoothing to $\bP^n$: first, 
Equation \eqref{Eqn:neccessCond-QGsmoothingPn} must hold, that is
\[ (n+1)^n (a_0 \dots a_n) = (a_0 + \dots + a_n)^n ,\] and second, all singularities are in codimension 2, i.e.
\begin{center}
    every subset of $k$ weights with a non-trivial common factor $d$ is contained in a set of $n-1$ weights with a non-trivial common factor. 
\end{center}

Using a computer program to list tuples of weights satisfying these conditions for $n = 3, 4$, we obtain the following weighted projective spaces that are \textit{candidates} to admit $\bQ$-Gorenstein smoothings to $\bP^3$ and $\bP^4$: 

\begin{proposition}
    A well-formed weighted projective space $\bP(a_0, a_1, a_2, a_3)$ admitting a smoothing to $\bP^3$, with $a_i \le 800$ for each $i$, must be contained in the following list: 
    \begin{align*}
        &(1, 1, 1, 1), (1, 1, 2, 4), (1, 2, 9, 12), (1, 4, 10, 25), (1, 6, 9, 32), (1, 9, 50, 60), (1, 20, 75, 144), 
        \\ &(1, 22, 32, 121), (1, 25, 54, 160), (1, 25, 65, 169), (1, 49, 160, 350), (1, 50, 289, 340), 
        \\& (1, 54, 121, 352), 
        (1, 72, 375, 512), (1, 160, 175, 784), (2, 9, 121, 132), (2, 15, 100, 243), 
        \\& (2, 25, 288, 405), 
        (2, 36, 57, 361), (3, 4, 63, 98), (3, 32, 45, 400), (5, 6, 9, 100), (10, 25, 49, 756)
    \end{align*}   
\end{proposition}

\begin{proposition}
    A well-formed weighted projective space $\bP(a_0, a_1, a_2, a_3, a_4)$ 
    that is smoothable
    to $\bP^4$, with $a_i \le 200$ for each $i$, must be contained in the following list: 
    \begin{align*}
        &(1, 1, 1, 1, 1), (1, 1, 2, 2, 4), (1, 2, 6, 9, 12), (1, 4, 10, 10, 25), (1, 6, 9, 12, 32), (1, 9, 30, 50, 60), \\& (1, 10, 25, 54, 60), (1, 12, 32, 54, 81), (1, 20, 60, 75, 144), (1, 22, 32, 44, 121), (1, 22, 54, 121, 132), \\& (1, 24, 45, 80, 150), (1, 25, 40, 90, 144), (1, 25, 54, 60, 160), (1, 25, 65, 65, 169), (1, 32, 54, 81, 192), \\&(2, 3, 3, 4, 18), (2, 6, 27, 64, 81), (2, 8, 75, 80, 135), (2, 9, 16, 72, 81), (2, 9, 64, 75, 150), \\&(2, 9, 66, 121, 132), (2, 11, 22, 64, 121), (2, 12, 21, 49, 126), (2, 15, 15, 18, 100), (3, 3, 4, 18, 32), \\&(3, 3, 18, 32, 64), (3, 4, 42, 63, 98), (3, 5, 72, 100, 120), (3, 7, 18, 84, 98), (3, 7, 98, 144, 168), \\&(3, 10, 32, 75, 180), (3, 10, 45, 50, 192), (3, 12, 25, 80, 180), (3, 14, 18, 28, 147), (4, 5, 45, 96, 150), \\&(5, 5, 32, 108, 150), (5, 6, 9, 30, 100), (5, 6, 25, 120, 144), (5, 9, 16, 120, 150), (5, 10, 18, 75, 192), \\& (6, 7, 14, 36, 147)
    \end{align*}   
\end{proposition}

The remainder of this paper will be focused on finding criteria for which these spaces admit smoothings.  The first tuple in either list is just $\bP^n$, and the second always admits a smoothing by the following example. 

\begin{example}
\label{ex:VeroneseEmbeddingTrick}
Consider the weighted projective space $\mathbb{P}(1, 1, 2_k, 4)$ where $k \geq 0$. For any $k \ge 0$, this is smoothable to $\bP^{k+2}$.  First, note that we have a `Veronese' embedding of $\mathbb{P}(1, 1, 2_k, 4)$ into $\mathbb{P}(1, 1, 1, 1_k, 2)$ given by 
\begin{center}
$[x_{0}: x_{1}: x_{2}: \dots :x_{k+1}: x_{k+2}] \mapsto [x_{0}^{2}: x_{0}x_{1}: x_{1}^{2}: x_{2}: \dots : x_{k+1}: x_{k+2}]$.
\end{center}
Let $[y_{0}: \dots : y_{k+3}]$ be the coordinates on $\mathbb{P}(1, 1, 1, 1_k, 2)$.  The image of the embedding is then defined by the degree 2 equation $y_{0}y_{2} - y_{1}^{2} = 0$.  The degree of the equation defining the image of the embedding matches the last weight of $\mathbb{P}(1, 1, 1, 1_k, 2)$, and the equation represents the central fiber of the family $y_{0}y_{2} - y_{1}^{2} = ty_{k+3}$ over $\bA^1_t$.  Because each fiber of this family is isomorphic to $\mathbb{P}(1, 1, 1, 1_k) \cong \mathbb{P}^{k+2}$ for $t \neq 0$, we have obtained a smoothing to $\mathbb{P}^{k+2}$.  This is furthermore $\bQ$-Gorenstein because the total space of the smoothing is a hypersurface in $\mathbb{P}(1, 1, 1, 1_k, 2) \times \bA^1$.
\end{example}

\section{A cone construction in families}\label{sec:cone}
In this section, we construct weighted projective degenerations of $\bP^{n+1}$ from weighted projective degenerations of $\bP^n$. 
To do this, we use the following construction (cf. \cite[3.8]{Kollar13}, or \cite[Definition 3.2]{Zh24}). 

\begin{definition}
\label{def:affineAndProjectiveCones}
Let $X$ be a projective scheme and $L$ an ample $\Q$-line bundle on $X$. We define the {\it affine orbifold cone over $X$ with respect to $L$} to be
\begin{center}
$C_{a}(X, L) := \Spec_{\C} \sum_{m \geq 0} H^{0}(X, L^{[m]})$,
\end{center}
and the {\it projective orbifold cone over $X$ with respect to $L$} to be
\begin{center}
$C_{p}(X, L) := \Proj_{\C}\:\sum_{m \geq 0} \left(\sum_{r=0}^{m} H^{0}(X, L^{[r]}) \cdot x_{n+1}^{m-r}\right)$.
\end{center}
\end{definition}

\begin{example}
\label{Ex:ProjectiveConePn}
Let $X = \mathbb{P}^{n}$ with line bundle $L = \mathcal{O}(1)$. The corresponding projective orbifold cone is
\begin{center}
$C_{p}(\mathbb{P}^{n}, \mathcal{O}(1)) = \Proj\:\sum_{m \geq 0} \left(\sum_{r=0}^{m} H^{0}(\mathbb{P}^{n}, \mathcal{O}(r)) \cdot x_{n+1}^{m-r}\right) = \Proj\:\C[x_{0}, \dots, x_{n+1}] = \mathbb{P}^{n+1}$.
\end{center}
\end{example}

\begin{example}
\label{Ex:ProjectiveConeWPS}  Let $X = \bP(a_0, \dots, a_{n})$ be a weighted projective space and let $L =\mathcal{O}(a_{n+1})$ for some positive integer $a_{n+1}$. As $X = \Proj\:\C[x_0, \dots, x_n]$ with grading given by $\deg x_i = a_i$, we have
\begin{center}
$C_{p}(X, \mathcal{O}(a_{n+1})) = \Proj\:\sum_{m \geq 0} \left(\sum_{r=0}^{m} H^{0}(X, \mathcal{O}(ra_{n+1})) \cdot x_{n+1}^{m-r}\right) = \Proj\:\C[x_{0}, \dots, x_{n+1}],$
\end{center}
where the grading on this polynomial ring is $\deg x_i = a_i$, $1\le i \le n+1$. In other words, 
\begin{center}
$C_{p}\left(\mathbb{P}(a_{0}, \dots, a_{n}), \mathcal{O}(a_{n+1})\right) = \mathbb{P}(a_{0}, \dots, a_{n}, a_{n+1})$.
\end{center}
\end{example}

We also understand the singularities of the cones constructed in this way. 

\begin{proposition}{\cite[Lemma 3.3]{Zh24}, \cite[Section 3]{Kollar13}}\label{prop:singsofcones}
    Let $X$ be a normal, projective scheme and $L$ an ample $\bQ$-line bundle on $X$ such that $K_X \equiv rL$ for some $r \in \bQ$.  Then, 
        \begin{enumerate}
            \item $C_p(X,L)$ is terminal if and only if $r < -1$ and $X$ is terminal;
            \item $C_p(X,L)$ is canonical if and only if $r \le -1$ and $X$ is canonical; 
            \item $C_p(X,L)$ is klt if and only if $r < 0$ and $X$ is klt;
            \item $C_p(X,L)$ is Gorenstein if $ r \in \bZ$ and $K_X$ is Cartier.
        \end{enumerate}
\end{proposition}

We next define cones in the relative setting.

\begin{definition}
Let $\pi: \mathcal{X} \rightarrow T$ be a flat family of varieties over a smooth curve $T$ and $\cL$ a relatively ample $\Q$-line bundle on $\mathcal{X}/T$.  We define the {\it relative projective cone} to be
\begin{center}
$\calC_{p}(\mathcal{X}/T, \mathcal{L}) = \Proj_{T}\:\sum_{m \geq 0}\left(\sum_{r=0}^{m} \pi_{\ast}(\mathcal{L}^{[r]}) \cdot x_{n+1}^{m-r}\right)$.
\end{center}
\end{definition}

\begin{lemma}
\label{Lemma:relativeProjectiveCone}
Assume that $\pi: \mathcal{X} \to T$ is a $\Q$-Gorenstein family of klt Fano varieties of dimension $n$ over a smooth curve $T$.  Let $\mathcal{L}$ be a relatively ample $\Q$-line bundle on $\mathcal{X}$ such that $\mathcal{O}(dK_{\mathcal{X}/T}) \sim_T \mathcal{L}^{ad}$ for some integer $d >0$ and $a \in \Q$ such that $dK_{\mathcal{X}/T}$ is Cartier.  Then, $\calC_{p}(\mathcal{X}/T, \mathcal{L})$ as defined above is a $\bQ$-Gorenstein family of klt Fano varieties of dimension $n+1$ over $T$.  The fiber over a point $t \in T$ of $\mathcal{C}_{p}(\mathcal{X}/T, \mathcal{L})$ is the projective orbifold cone \[ C_{p}(\mathcal{X}_t, \mathcal{L}_t) = \Proj\:\sum_{m \geq 0}\biggl(\sum_{r=0}^{m} H^0(\mathcal{X}_t, \mathcal{L}_t^{[r]}) \cdot x_{n+1}^{m-r}\biggr)\] as defined above.
\end{lemma}

\begin{proof}
    Because $\mathcal{L}$ is a relatively ample $\Q$-line bundle, for each closed point $t \in T$, we have $H^i(\mathcal{X}_t, \mathcal{L}_t^{[m]}) = 0$ for all $m \in \mathbb{Z}$ and all $0 < i < \dim \mathcal{X}_t$.  Indeed, by Kodaira vanishing, this is clear for $m < 0$.  For $m \ge 0$, as $\mathcal{X}_t$ is Fano and Cohen Macaulay and $\mathcal{L}_t^{[-m]}(K_X)$ is antiample, the vanishing follows from Serre duality and Kodaira vanishing. 

    By cohomology and base change, the $i = 1$ case gives that $\pi_* \mathcal{L}^{[r]} $ is a vector bundle over $T$ whose fibers over $t \in T$ are precisely $H^0(\mathcal{X}_t, (\mathcal{L}_t)^{[r]})$.  Therefore, by \cite[Lemma 31.30.6]{stacksproject}, $\mathcal{C}_{p}(\mathcal{X}/T, \mathcal{L})$ is flat over $T$ and, as relative $\Proj$ commutes with base change, we have
    \begin{align*}
    \mathcal{C}_{p}(\mathcal{X}/T, \mathcal{L})_t &\cong \Proj\:\sum_{m \geq 0}\biggl(\sum_{r=0}^{m} \pi_{\ast} \mathcal{L}^{[r]} \cdot x_{n+1}^{m-r}\biggr)\Bigg|_t \\
    &\cong  \Proj\:\sum_{m \geq 0}\biggl(\sum_{r=0}^{m} H^0(\mathcal{X}_t, \mathcal{L}_t^{[r]}) \cdot x_{n+1}^{m-r}\biggr) \\
    &= C_p(\mathcal{X}_t, \mathcal{L}_t).
    \end{align*}

    By \cite[Lemma 3.3]{Zh24} (see also \cite[Proposition 3.14]{Kollar13}) and the assumption that $\mathcal{X}_t$ is klt Fano, it follows that $C_{p}(\mathcal{X}_t, \mathcal{L}_t)$ is klt.  Finally, as the intersection numbers $(K_{C_{p}(\mathcal{X}_t, \mathcal{L}_t)})^{n+1}$ are constant (\cite[Proposition 2.21]{ABB+}), by \cite[Theorem 5.8]{Kollar23}, this implies that the family $\mathcal{C}_{p}(\mathcal{X}/T, \mathcal{L})$ is $\Q$-Gorenstein.
\end{proof}

We can also study families of divisors in this construction. 

\begin{lemma}\label{lem:coneofdivisors}
    Let $\pi: \calX \to T$ be a $\bQ$-Gorenstein family of klt Fano varieties over a smooth curve $T$ and $\calL$ a relatively ample $\bQ$-line bundle on $\calX$.  Let $\calD \subset \calX$ be a flat family of divisorial subschemes \cite[Definition-Lemma 4.19]{Kollar23} such that $\calO(a\calD) \sim_T \calL^{[b]}$ for some $a,b \in \bZ_{>0}$.  Then, $\calC_p(\calD/T, \calL_\calD)$ is a flat family of divisorial subschemes of $\calC_p(\calX/T, \calL)$ whose fiber over $t \in T$ is $C_p(\calD_t, \calL_{\calD_t}) \subset C_p(\calX_t, \calL_t)$.
\end{lemma}

\begin{proof}
    For $r \ge 0$, consider the exact sequence 
    \[ 0 \to \calL^{[r]}(-\calD) \to \calL^{[r]} \to \calL^{[r]}_\calD\to 0\] and its pushforward 
    \[ 0 \to \pi_*\calL^{[r]}(-\calD) \to \pi_*\calL^{[r]} \to \pi_*\calL^{[r]}_\calD\to R^1\pi_*\calL^{[r]}(-\calD) .  \]
    For each closed point $t \in T$, if $ar - b < 0$, $\calL_t^{[r]}(-\calD_t)$ is antiample, so $H^i(\calX_t, \calL_t^{[r]}(-\calD_t)) = 0$ for $i > 0$ by Kodaira vanishing.  If $ar -b \ge 0$, as $\calX_t$ is klt Fano and Cohen Macaulay, $H^i(\calX_t, \calL_t^{[r]}(-\calD_t))$ is Serre dual to $H^{n-i}(\calX_t, \calL_t^{[-r]}(\calD_t)(K_{\calX_t}))$ and $\calL_t^{[-r]}(\calD_t)(K_{\calX_t})$ is antiample so this cohomology group vanishes again by Kodaira vanishing.  Therefore, by cohomology and base change, $R^1\pi_*\calL^{[r]}(-\calD)  = 0$ and there is a surjection
    \[  \pi_*\calL^{[r]} \to \pi_*\calL^{[r]}_\calD\to 0.  \]

    This gives an injection $\calC_p(\calD/T, \calL_\calD)\subset \calC_p(\calX/T, \calL)$, so $\calC_p(\calD/T, \calL_\calD)$ is in fact a subscheme of $ \calC_p(\calX/T, \calL)$. 

    Now, consider $\calC_p(\calD/T, \calL_\calD)$.  For any $t \in T$, we claim $H^1(\calD_t, \calL_{\calD_t}^{[r]}) = 0$.  Indeed, we have an exact sequence of sheaves on $X_t$
    \[ 0 \to \calL_t^{[r]}(-\calD_t) \to \calL_t^{[r]} \to \calL_{\calD_t}^{[r]} \to 0\] and the result follows from the vanishing $H^i(\calX_t, \calL_t^{[r]}(-\calD_t)) = 0$ from above and the vanishing $H^i(\calX_t, \calL_t^{[r]})=0$ from the previous proof.  By cohomology and base change, $\pi_*\calL^{[r]}_\calD$ is a vector bundle over $T$ whose fibers are $H^0(\calD_t, \calL_{D_t}^{[r]})$ so $\calC_p(\calD/T, \calL_\calD)$ is flat over $T$ by \cite[Lemma 31.30.6]{stacksproject} and its fibers are precisely $C_p(\calD_t, \calL_{\calD_t})$.  This is then a flat family of divisorial subschemes by \cite[Definition-Lemma 4.19]{Kollar23}.
\end{proof}

Finally, we combine the previous two results (\Cref{Lemma:relativeProjectiveCone} and \Cref{lem:coneofdivisors}) into the main theorem of this section. 

\begin{theorem}\label{thm:conesanddivisors}
    Suppose $X,Y$ are klt Fano varieties such that $Y$ is a $\bQ$-Gorenstein degeneration of $X$, i.e. there is a $\bQ$-Gorenstein family $\calX \to T$ over a smooth curve with special fiber $\calX_0 = Y$ and general fiber $\calX_t = X$.  Let $\calL$ be a relatively ample $\bQ$-line bundle on this family such that $\calL$ restricts to a multiple of the anticanonical divisor on each fiber. Then,
        \begin{enumerate}
            \item The cone
            $C_p(Y, \calL|_Y)$ is a $\bQ$-Gorenstein degeneration of $C_p(X, \calL|_X)$, and both $C_p(Y, \calL|_Y)$ and $C_p(X, \calL|_X)$ are klt Fano varieties.
            \item If $D_Y \subset Y$ deforms to a divisor $D_X \subset X$ in this degeneration, i.e. there exists a flat family of divisorial subschemes $\calD$ on $\calX$ such that $\calD_0 = D_Y$ and $\calD_t = D_X$ and $\calO(\calD)$ is $T$-equivalent to a multiple of $\calL$, then $C_p(D_Y, \calL|_{D_Y})$ is a degeneration of $C_p(D_X, \calL|_{D_X})$ inside the degeneration $C_p(X, \calL|_X) \rightsquigarrow C_p(Y, \calL|_Y)$. 
            \item If $\Delta_Y \in |\calO(C_p(D_Y, \calL|_{D_Y}))|$ is any effective integral divisor on $C_p(Y, \calL|_Y)$ linearly equivalent to $C_p(D_Y, \calL|_{D_Y})$, then $\Delta_Y$ deforms to a divisor $\Delta_X \in |\calO(C_p(D_X, \calL|_{D_X}))|$ on $C_p(X, \calL|_X) $. 
        \end{enumerate}
\end{theorem}

\begin{proof}
    Statement (1) follows from Lemma \ref{Lemma:relativeProjectiveCone}.  Statement (2) follows from Lemma \ref{lem:coneofdivisors}.  Finally, for statement (3), note that (2) gives a deformation of the ample $\bQ$-line bundle $\calO(C_p(D_Y, \calL|_{D_Y}))$ to $\calO(C_p(D_X, \calL|_{D_X}))$.  To show that any section of this $\bQ$-line bundle deforms in this family, we verify that $H^1(C_p(Y, \calL|_Y), \calO(C_p(D_Y, \calL|_{D_Y}))) = 0$, which follows because $C_p(Y, \calL|_Y)$ is a klt Fano variety. 
\end{proof}

In the remainder of this section, we collect several applications of this result, including the proof of \Cref{thm:infiniteFamiliesDegenerationsPn}. 

First, we explain an application of \Cref{thm:conesanddivisors} which will give infinitely many degenerations of $\mathbb{P}^{n}$. Suppose that there exists a $\mathbb{Q}$-Gorenstein smoothing of $\mathbb{P}(a_{0}, \dots, a_{n-1})$ to $\mathbb{P}^{n-1}$.  Because this is $\bQ$-Gorenstein, the canonical divisor $\calO(K_{\mathbb{P}^{n-1}}) = \calO(-n)$ must degenerate to $\calO(K_{\mathbb{P}(a_{0}, \dots, a_{n-1})}) = \mathcal{O}(-a_{0} - \dots - a_{n-1})$. Equivalently, the line bundle $\mathcal{O}(1)$ on $\mathbb{P}^{n-1}$ degenerates to the $\bQ$-line bundle $\mathcal{O}(a_n)$ on the weighted projective space $\mathbb{P}(a_{0}, \dots, a_{n-1})$, where
\begin{center}
$a_n := \displaystyle{\frac{a_{0} + \dots + a_{n-1}}{n}}$,
\end{center}
an integer by Proposition \ref{prop:relationsonweights}. Note that the limit of $\mathcal{O}(1)$ is determined by the degree of the line bundle as the central fiber has Picard rank one.  

Now, as a consequence of \Cref{thm:conesanddivisors},
we obtain
that $C_{p}(\mathbb{P}^{n-1}, \mathcal{O}(1))$ degenerates to $C_{p}(\mathbb{P}(a_{0}, \dots, a_{n-1}), \mathcal{O}(a_n))$.  By \Cref{Ex:ProjectiveConePn} and \Cref{Ex:ProjectiveConeWPS}, it follows that $\mathbb{P}^{n}$ degenerates to $\mathbb{P}(a_{0}, \dots, a_{n-1}, a_{n})$. We summarize this result in the following corollary.

\begin{corollary}\label{cor:degenofPn+1}
Suppose that $\mathbb{P}(a_{0}, \dots, a_{n-1})$ admits a $\mathbb{Q}$-Gorenstein smoothing to $\mathbb{P}^{n-1}$. Then,  $\mathbb{P}(a_{0}, \dots, a_{n-1}, a_{n})$ admits a $\mathbb{Q}$-Gorenstein smoothing to $\mathbb{P}^{n}$, where $ a_{n} := \frac{\sum_{i=0}^{n-1} a_{i}}{n}$.
\end{corollary}

The previous corollary gave a global smoothing of 
$\bP(a_0, \dots, a_n)$ to $\bP^n$.  
In particular,
this implies that every cyclic quotient singularity on such $\bP(a_0, \dots, a_n)$ is smoothable.  

\begin{corollary}\label{cor:localsmoothings}
    A cyclic quotient singularity of the form $\frac{1}{a_n}(a_0, a_1, \dots, a_{n-1})$, where the weighted projective space $\bP(a_0, a_1, \dots, a_{n-1})$ is smoothable to $\bP^n$ and $a_{n} := \frac{\sum_{i=0}^{n-1} a_{i}}{n}$, is smoothable.  
\end{corollary}

This gives examples of Gorenstein non-isolated cyclic quotient singularities that are smoothable. 

\begin{example}\label{ex:alocalsmoothings}
    By \Cref{thm:wpdegensofP2}, the weighted projective space $\bP(1,4,25)$ is smoothable to $\bP^2$.  Because $10 = \frac{1+4+25}{3}$, the Gorenstein cyclic quotient singularity $\frac{1}{10}(1,4,25) = \frac{1}{10}(1,4,5)$ 
    is smoothable by Corollary \ref{cor:localsmoothings}.    
\end{example}

The results from this section prove the first part of \Cref{thm:infiniteFamiliesDegenerationsPn}.

\begin{theorem}
\label{thm:degensofP2toPn}
For $n \ge 2$, let $X =\bP(a^{2}, b^{2}, c^{2}, abc_{n-2})$ where $a^2 + b^2 + c^2 = 3abc$.  Then $X$ admits a $\mathbb{Q}$-Gorenstein smoothing to $\mathbb{P}^{n}$.
\end{theorem}

\begin{proof}
    By \Cref{thm:wpdegensofP2}, $\bP(a^{2}, b^{2}, c^{2})$ admits a $\bQ$-Gorenstein smoothing to $\bP^2$.
    Since the average of the weights is precisely $abc$,
    using \Cref{cor:degenofPn+1} and induction we see that $\bP(a^{2}, b^{2}, c^{2},abc,\ldots,abc)$ also admits a $\bQ$-Gorenstein smoothing to $\bP^n$.
\end{proof}

\begin{remark}
    Applied to any of the partial smoothings of $\bP(a^2,b^2,c^2)$, this cone construction also yields degenerations of $\bP^n$ that are orbifold cones over partial smoothings.  
\end{remark}

We can also apply the cone construction to other degenerations of $\bP^n$. 
First, we include a well-known result for convenience of the reader. 

\begin{theorem}
\label{thm:degree-d-hypersurfaces}
Suppose $X_{d} \subset \mathbb{P}^{n}$ is a degree $d$ hypersurface. Then there exists a $\bQ$-Gorenstein degeneration of $\mathbb{P}^{n}$ to the cone $C_p(X_{d}, \mathcal{O}(d))$.
\end{theorem}

\begin{proof}
    Consider the degree $d$ embedding $|\calO(d)|: \bP^{n} \to \bP^N$. 
    This is projectively normal and $X_d \cong H$ is a hyperplane section of the image. 
 Therefore, by \cite[Example 7.61]{KM98}, $\bP^n$ admits a degeneration to the cone over $X_d$.  By \cite[\S 3]{Kollar13}, this cone is precisely $C_p(X_d, \calO(d))$.  By the construction in \cite[Example 7.61]{KM98}, the total space $\calX$ of this degeneration is a blow up of a smooth subvariety in $C_p(\bP^n, \calO(d))$, which is $\bQ$-Gorenstein.  
\end{proof}

This gives the following: 

\begin{corollary}
\label{cor:hypersurfaceDegeneration}
Suppose that
there is
a Fano hypersurface $X_{d}$ of degree $d$ in $\mathbb{P}^{n}$ 
that
degenerates to $\mathbb{P}(a_{0}, \dots, a_{n-1})$. Then $\mathbb{P}^{n}$ degenerates to $\mathbb{P}(a_{0}, \dots, a_{n-1}, da_{n})$, where $a_{n} = \frac{a_0 + \dots + a_{n-1}}{n+1-d} \in \bZ$ is the degree of the limit of $\mathcal{O}_{X_d}(1)$ on the weighted projective space $\bP(a_0, \dots, a_{n-1})$.
\end{corollary}

\begin{proof}
From \Cref{thm:degree-d-hypersurfaces}, there exists a degeneration of $\mathbb{P}^{n}$ to $C_p(X_d, \calO(d))$. As $X_d$ admits a degeneration to $\bP(a_0, \dots, a_{n-1})$, by Theorem \ref{thm:conesanddivisors}, $C_p(X_d, \calO(d))$ admits a degeneration to $C_{p}(\mathbb{P}(a_{0}, \dots, a_{n-1}, da_{n}), \mathcal{O}(da_{n})) = \bP(a_0, \dots, a_{n-1}, da_{n})$.
It then follows from Lemma \ref{lem:xtoytoz} that $\bP^{n}$ admits a degeneration to $\bP(a_0, \dots, a_{n-1}, da_{n})$.

The weight $a_n$ can be explicitly computed as follows: by construction, the degeneration $X_d \rightsquigarrow \bP(a_0, \dots a_{n-1})$ is $\bQ$-Gorenstein, so $\calO(-K_{X_d})^{n-1} = \calO_{X_d}(n+1-d)^{n-1} = d(n+1-d)^{n-1}$.  This must equal $\calO(-K_{\bP(a_0,\dots,a_{n-1})})^{n-1} = \frac{(a_0 + \dots + a_{n-1})^{n-1}}{a_0\dots a_{n-1}}$, so we have 
\[ \frac{a_0 + \dots + a_{n-1}}{n+1-d} = \sqrt[n-1]{da_0\dots a_{n-1} }\]
and by the rational root theorem conclude $\frac{a_0 + \dots + a_{n-1}}{n+1-d} \in \bZ$. 

Denote by $\calO(a_{n})$ the limit of $\calO_{X_d}(1)$.  As the $(n+1-d)$th reflexive tensor power of these $\bQ$-line bundles must agree with the canonical, we have
\[ (n+1-d)a_{n} = a_0 + \dots + a_{n-1}\] and therefore 
\[ a_{n} = \frac{a_0 + \dots + a_{n-1}}{n+1-d},\]
which is an integer by the previous paragraph.
\end{proof}

This also has a local consequence:

\begin{corollary}\label{cor:localsmoothings2}
    A cyclic quotient singularity $\frac{1}{da_n}(a_0, a_1, \dots, a_{n-1})$, where the weighted projective space $\bP(a_0, a_1, \dots, a_{n-1})$ is smoothable to a degree $d$ hypersurface in $\bP^n$ and $a_n := {\frac{\sum_{i=0}^{n-1} a_{i}}{n+1-d}}$, is smoothable.  
\end{corollary}

\begin{example}
\label{ex:P(2,12,21,49)-ps}
    Consider a cyclic quotient singularity of type $\frac{1}{4}(1,1,2)$ or $\frac{1}{12}(1,2,9)$.  Observe that $4 = 1+ 1+ 2 = 2(\frac{1+1+2}{4-2})$ and $12 =1+ 2 + 9  = 2(\frac{1+2+9}{4-2})$.  By the following \Cref{thm:HP-quadricSurfaceP3} of Hacking--Prokhorov,
    the weighted projective spaces $\bP(1,1,2)$ and $\bP(1,2,9)$ are smoothable to $\bP^1 \times \bP^1$ (a quadric surface in $\bP^3$), so these cyclic quotient singularities are smoothable by Corollary \ref{cor:localsmoothings2}. 
\end{example}

By the following description of weighted projective degenerations of the quadric surface, Corollary \ref{cor:hypersurfaceDegeneration} will yield infinitely many additional weighted projective degenerations of $\bP^n$ for $n \ge 3$.

\begin{theorem}
{\cite{HP10}}
\label{thm:HP-quadricSurfaceP3}
 A weighted projective surface admits a $\bQ$-Gorenstein smoothing to a smooth quadric surface in $\mathbb{P}^{3}$ if and only if it is of the form $\mathbb{P}(a^{2}, b^{2}, 2c^{2})$, where the weights satisfy $a^{2} + b^{2} + 2c^{2} = 4abc$.
\end{theorem}

\begin{example}\label{ex:quadricdegensofP3}
By the previous theorem, 
every weighted projective space of the form $\mathbb{P}(a^{2}, b^{2}, 2c^{2})$, such that $a^{2} + b^{2} + 2c^{2} = 4abc$, is a $\bQ$-Gorenstein degeneration of a smooth quadric surface $X\subset\bP^3$.
Now \Cref{cor:hypersurfaceDegeneration} implies that there exists a $\bQ$-Gorenstein degeneration of $\mathbb{P}^{3}$ to $\mathbb{P}(a^{2}, b^{2}, 2c^{2}, 2a_{3})$, where $a_{3}$ is the degree of the limit of $\mathcal{O}_X(1)$ on $\mathbb{P}(a^{2}, b^{2}, 2c^{2})$, which is 
\[ a_3 = \frac{a^2 + b^2 + 2c^2}{4 - 2} = \frac{4abc}{2} = 2abc.\]
Therefore, $\mathbb{P}^{3}$ degenerates to $\mathbb{P}(a^{2}, b^{2}, 2c^{2}, 4abc)$.
\end{example}

 Now we can prove the second part of Theorem \ref{thm:infiniteFamiliesDegenerationsPn}. 

 \begin{theorem}
\label{thm:degensofP1xP1toPn}
For $n \ge 3$, let $X =\bP(a^{2}, b^{2}, 2c^{2}, 2abc_{n-3}, 4abc)$ where $a^2 + b^2 + 2c^2 = 4abc$.  Then $X$ admits a $\mathbb{Q}$-Gorenstein smoothing to $\mathbb{P}^{n}$.
\end{theorem}

\begin{proof}
    By \Cref{ex:quadricdegensofP3}, $\bP(a^2,b^2,2c^2,4abc)$ admits a $\bQ$-Gorenstein smoothing to $\bP^3$. The average of the weights is $\frac{a^2+b^2+2c^2+4abc}{4}=2abc$, so \Cref{cor:degenofPn+1} and induction give the result.
\end{proof}

\section{Deformations of weighted projective threefolds}\label{sec:deformations}

In this section, we prove that a weighted projective threefold with mild assumptions that is locally smoothable everywhere is in fact globally smoothable.  This will provide an additional tool for smoothing weighted projective spaces that are not cones over lower dimensional smoothable weighted projective spaces.

Throughout this section, we will make the following assumptions.

\begin{assumption}\label{assumptions-on-wps}
Let $X = \bP(a_0,a_1,a_2,a_3)$ be a well-formed weighted projective threefold and assume that $X$ has only singularities of type $A_n$ in codimension 2, which necessarily occur along the curves $C_{ij} = (x_i = x_j = 0)$.  Equivalently, this condition is precisely that $a_k + a_l \equiv 0 \mod (\gcd(a_i,a_j))$ for every permutation $(i,j,k,l)$ of $(0,1,2,3)$ (so that $X$ has $A_{\gcd(a_i,a_j)-1}$ singularities along $C_{ij}$). Assume further for each $C_{ij}$ such that $X$ is singular at the generic point of $C_{ij}$ that there exists a section $s \in H^0(X, \mathcal{\omega}_X^{\vee})$ that does not generically vanish along $C_{ij}$.  Such a section exists, for example, if $2a_l + 2a_k + a_i + a_j > \frac{a_ia_j}{\gcd(a_i,a_j)}$.
\end{assumption}

\begin{remark}
    These assumptions are satisfied for the two families of weighted projective threefolds of the form $\bP(a^2,b^2,c^2,abc)$ such that $a^2 + b^2 + c^2 = 3abc$ and $\bP(a^2,b^2,2c^2, 4abc)$ such that $a^2 +b^2 +2 c^2 = 4abc$.  The one-dimensional singular strata are $C_{i3}$ for $i = 0,1,2$ in both cases.  For the second assumption, note that the anticanonical divisor has degree $4abc$ in the first case (respectively, $8abc$ in the second case), and the section $x_3^4$ (respectively, $x_3^2$) does not vanish generically along the one-dimensional singular strata.  They are also satisfied for any Gorenstein weighted projective threefold.
\end{remark}

Write $X = \cup_{i = 0}^3 U_i$, where $U_i = (x_i \ne 0 )$.  The goal of this section is to prove that, given any deformations of the sets $U_i$ that agree on the intersections $U_i \cap U_j$, they glue together to form a deformation of $X$.  The deformation space of $X$ is given by $\Ext^1(\Omega_X, \calO_X)$ and the obstruction space is given by $\Ext^2(\Omega_X, \calO_X)$.  We will utilize the following exact sequence from the local-to-global spectral sequence for $\Ext$, where $T_X = \SHom(\Omega_X,\calO_X)$: 

\begin{equation}
\begin{split} \label{eqn:localtoglobalExt}
    0 &\to H^1(X, T_X) \to \Ext^1(\Omega_X, \calO_X) \to H^0(X, \SExt^1(\Omega_X, \calO_X)) \to H^2(X,T_X) \\
    &\to K \to H^1(X, \SExt^1(\Omega_X, \calO_X)) \to H^3(X,T_X)
\end{split}
\end{equation}

where $K = \ker \ (\Ext^2(\Omega_X, \calO_X) \to H^0(X, \SExt^2(\Omega_X, \calO_X))) $.

We begin with some preliminary results. 

\begin{lemma}\label{lem:AnExt2}
    If $V$ is a surface with only $A_n$ singularities, then  $\SExt^2(\Omega_V,\calO_V) = 0$.  Consequently, if $X$ is a weighted projective space with only $A_n$ singularities in codimension 2, then $\SExt^2(\Omega_X,\calO_X)$ is supported in codimension 3.  Explicitly, writing $X = \bP(a_0,a_1,a_2,a_3)$, then $\SExt^2(\Omega_X,\calO_X)$ is supported on the singular points $[1:0:0:0], [0:1:0:0], [0:0:1:0]$, and $[0:0:0:1]$.
\end{lemma}

\begin{proof}
    Let $V$ be a surface with only $A_n$ singularities.  It suffices to work locally in a neighborhood of each $A_{n}$ singularity, each of which are of the form $(uv - t^{n+1} = 0) \subset \bA^3$.  Because this is a hypersurface singularity, there is a locally free resolution of $\Omega_V$ by only two terms 
    \[ 0 \to I/I^2 \to \Omega_{\bA^3}|_{V} \to \Omega_V \to 0\]
    and hence $\SExt^2(\Omega_V,\calO_V) = 0$.
\end{proof}

\begin{corollary}\label{cor:globaldefspaceislocal}
    For $X =  \cup_{i = 0}^3 U_{i}$ a weighted projective threefold that has only $A_n$ singularities in codimension 2, the vector space
    \begin{center}
    $H^0(X, \SExt^2(\Omega_X,\calO_X))$
    \end{center}
    is identified with the obstruction space for the local deformations of each $U_i$. 
\end{corollary}

\begin{proof}
    For each $U_i$, the obstruction space $\Ext^2(\Omega_{U_i}, \calO_{U_i}) = H^0(U_i, \SExt^2(\Omega_{U_i}, \calO_{U_i}))$ as $U_i$ is affine.  Because $\Ext^2(\Omega_{U_i}, \calO_{U_i})$ is supported only at the origin $(0,0,0) \in U_i$, we have 
    \[ H^0(X, \SExt^2(\Omega_X,\calO_X)) = \oplus_{i=0}^3H^0(U_i, \SExt^2(\Omega_{U_i}, \calO_{U_i})).  \]
\end{proof}

Combining the Sequence \ref{eqn:localtoglobalExt} with Theorem \ref{thm:cohomologyofwps} (implying that $H^i(X,T_X) = 0$ for $i \geq 1$) and Lemma \ref{lem:AnExt2}, we obtain the following:

\begin{corollary}\label{cor:wpssmoothing}
    Let $X$ be a well-formed weighted projective threefold such that $X$ has only $A_{n}$ singularities in codimension 2.  Then, 
        \begin{enumerate}
            \item $\Ext^1(\Omega_X, \calO_X) \cong H^0(X, \SExt^1(\Omega_X, \calO_X))$ and 
            \item there is an exact sequence 
            \[ 0 \to H^1(X, \SExt^1(\Omega_X, \calO_X)) \to \Ext^2(\Omega_X, \calO_X) \to H^0(X, \SExt^2(\Omega_X, \calO_X)) .\]
        \end{enumerate}
        
    If $H^1(X, \SExt^1(\Omega_X, \calO_X)) = 0$, this gives an injection on obstruction spaces
    \begin{center}
    $\Ext^2(\Omega_X, \calO_X) \hookrightarrow H^0(X, \SExt^2(\Omega_X, \calO_X))$.
    \end{center}
    In particular, if a deformation of each $U_i$ exists and they agree on the overlaps $U_i \cap U_j$, these glue together to give a deformation of $X$.
\end{corollary}

\begin{proof}
    The only part that must be verified is the last sentence.  If a deformation of each $U_i$ exists and they agree on the overlaps, this gives an element of $H^0(X, \SExt^1(\Omega_X, \calO_X))$ which gives a first-order deformation in $\Ext^1(\Omega_X, \calO_X)$.  However, because the deformation of $U_i$ exists for each $i$, the corresponding obstruction in 
    \begin{center}
    $H^0(X, \SExt^2(\Omega_X,\calO_X)) = \oplus_{i=0}^3H^0(U_i, \SExt^2(\Omega_{U_i}, \calO_{U_i}))$
    \end{center}
    (the equality above coming from \Cref{cor:globaldefspaceislocal}) 
    vanishes and hence there is no obstruction in $\Ext^2(\Omega_X, \calO_X)$ to extending the deformation. 
\end{proof}

Now, we wish to verify that a weighted projective threefold satisfying \Cref{assumptions-on-wps} has $H^1(X, \SExt^1(\Omega_X, \calO_X)) = 0$.  The main tool will be to compare $\Omega_X$, the sheaf of Kahler differentials, with its double dual $\Omega_X^{[1]}$.  The sheaf $\Omega_X$ can be viewed as the differentials of invariant functions on $\bC[x_0, x_1, x_2, x_3 ]$, and by \cite{Kni73}, $\Omega_X^{[1]}$ is the sheaf of invariant differentials.  We will show that the cokernel of the natural map $\Omega_X \to \Omega_X^{[1]}$ determines this cohomology group and it can be explicitly computed. To do this, we begin with preparatory lemmas.  The first lemma is classically known but we include a proof that will be generalized to the case of interest.

\begin{lemma}\label{lem:extofAn}
    Let $V = \frac{1}{n+1}(1,n) = \bA^2_{(z,w)}/\mu_{n+1} $ be an $A_n$ singularity. Then, the cokernel of the double dual morphism $\Omega_V \to \Omega_V^{[1]}$ is a torsion sheaf at $(0,0)$ of length $n$ and can be identified with the $\mu_{n+1}$-invariant differential forms $(zw)^i z dw$ for $0 \le 1 \le n-1$.  This identification implies that $\Ext^1(\Omega^1_V, \calO_V) = \oplus_{i = 0}^{n-1} k$.
\end{lemma}

\begin{proof}
    We utilize two classical descriptions of $V$. As above, it can be expressed as the quotient $ \bA^2_{(z,w)}/\mu_{n+1}$, where $\mu_{n+1}$ acts by $\zeta_{n+1} \cdot (z,w) = (\zeta_{n+1} z, \zeta_{n+1}^{-1}w)$.  From this description, $V = \Spec  \ k [z,w]^{\mu_{n+1}}$.  Computing the invariant ring, $V$ can also be expressed as the Gorenstein hypersurface $uv - t^{n+1} = 0$ inside $\bA^3_{(u,v,t)}$, i.e. $V = \Spec \ k [u,v,t]/(uv - t^{n+1})$, where the isomorphism $k[u,v,t]/(uv - t^{n+1}) \to k[z,w]^{\mu_{n+1}}$ is given by $u \mapsto z^{n+1}$, $v \mapsto w^{n+1}$, $t \mapsto zw$.  Denote by $R = k [u,v,t]/(uv - t^{n+1})$.

    Consider the double dual morphism $\Omega^1_V \to \Omega_V^{[1]}$.  This is an isomorphism on the smooth locus of $V$ and by \cite[Proposition 9.7, Corollary 9.8]{Ku86}, $\Omega^1_V$ is torsion free, and hence $\Omega^1_V \to \Omega_V^{[1]}$ is injective.  Let $Q$ be the cokernel, yielding the exact sequence 
    \[ 0 \to \Omega^1_V \to \Omega_V^{[1]} \to Q \to 0.\]
    By definition, $\Omega^1_V$ is the $R$-module generated by the differential forms $du, dv, dt$ subject to the relation $vdu + u dv = (n+1)t^n dt$.  By \cite[Theorem 3]{Kni73}, the module $\Omega_V^{[1]}$ is precisely the $\mu_{n+1}$-invariant differential forms on $\bA^2$, i.e. $\Omega_V^{[1]}$ is the $R$-module generated by $z^{n} dz$, $w^{n}dw$, $zdw$ and $w dz$.  By the identification of $R$ with $k[z,w]^{\mu_{n+1}}$, the map $\Omega^1_V \to \Omega_V^{[1]}$ is given by $du \mapsto (n+1)z^n dz$, $dv \mapsto (n+1)w^n dw$, and $dt \mapsto zdw + wdz$.  It is then straightforward to verify that $Q$ is generated by the $\mu_{n+1}$-invariant differential forms $(zw)^i z dw$ for $0 \le i \le n-1$.  
    As a module, $Q$ is just $n$ copies of the field $k$ at the singular point of $V$, and therefore $\Ext^2(Q, \calO_V) = \oplus_{i = 0}^{n-1} k$.  From the exact sequence
    \[ 0 \to \Omega^1_V \to \Omega_V^{[1]} \to Q \to 0\]
    there is a long exact sequence of $\Ext$-groups 
    \[ \dots \to  \Ext^1(\Omega_V^{[1]}, \calO_V) \to \Ext^1(\Omega^1_V,\calO_V) \to \Ext^2(Q, \calO_V) \to \Ext^2(\Omega^{[1]}_V, \calO_V) \to \dots .\]
    By \cite[Theorem 3.3.10]{BH98}, as $\Omega^{[1]}_V$ is reflexive and $\dim V = 2$, it is Cohen-Macaulay, and $V$ is Gorenstein, and therefore $\Ext^i(\Omega^{[1]}_V, \calO_V) = \Ext^i(\Omega^{[1]}_V, \omega_V) \otimes \omega_V^\vee = 0$ for all $i > 0$. Therefore,
    \begin{center}
    $\Ext^1(\Omega^1_V,\calO_V) \cong  \Ext^2(Q, \calO_V)$,
    \end{center}
    proving the lemma.
\end{proof}

To show that $H^1(X, \SExt^1(\Omega_X, \calO_X)) = 0$, we will mimic the previous proof and compare $\Omega_X$ and $\Omega_X^{[1]}$.  However, if $X$ is not Gorenstein, we will not necessarily have vanishing of $\SExt^i(\Omega_X^{[1]}, \calO_X)$, so will work instead with $\SExt^i(\Omega_X^{[1]}, \omega_X)$ as it will still hold that $\SExt^i(\Omega_X^{[1]}, \omega_X) = 0$ for $i > 0$.  Indeed, from Theorem \ref{thm:wps-dualizingprops}, because $\Omega_X^{[1]}$ is Cohen-Macaulay whose support is all of $X$, the vanishing follows from \cite[Theorem 3.3.10]{BH98}.

To compare $\SExt^i(\Omega_X^{[1]}, \omega_X) $ with $\SExt^i(\Omega_X^{[1]}, \calO_X)$, we use the following: 

\begin{lemma}\label{lem:extdoubledual}
    Let $\calF$ be a sheaf on a projective variety and let $\calR$ be a reflexive sheaf.  Then, for all $i > 0$, there is a natural map $\SExt^i(\calF, \calR) \otimes \calR^\vee \to \SExt^i(\calF, \calO)$ which is an isomorphism where $\calR$ is locally free.
\end{lemma}

\begin{proof}
    By \cite[Proof of Lemma 15.66.3]{stacksproject}, there is a canonical map 
    \[ \SExt^i(\calF, \calR) \otimes \calR^\vee \to \SExt^i(\calF, \calR \otimes \calR^{\vee}).\]
    From the canonical map $\calR \otimes \calR^{\vee} \to \calO$ and the induced map on $\SExt$, we get a map 
    \[ \SExt^i(\calF, \calR) \otimes \calR^\vee \to \SExt^i(\calF, \calR \otimes \calR^{\vee}) \to \SExt^i(\calF, \calO).\]
    Finally, where $\calR$ is locally free, by \cite[Proposition III.6.7]{Hartshorne}, this is an isomorphism.
\end{proof}

\begin{proposition}\label{prop:computingextoncurves}
    Let $X$ be a well-formed weighted projective threefold $\bP(a_0, a_1,a_2,a_3)$ and let $C_{ij}$ be one of the one-dimensional singular strata of $X$ such that $X$ has an $A_n$ singularity at the generic point of $C_{ij}$.  Denote by $\{k,l\}$ the indices $\{0,1,2,3\} - \{i,j\}$.  For each $0 \le m \le n-1$, let $d_{i,m}$ be the smallest positive integer such that \[a_id_{i,m} + (m+1)(a_k + a_l) \equiv 0 \mod a_j\] and similarly let $d_{j,m}$ be the smallest positive integer such that  \[  a_jd_{j,m} + (m+1)(a_k + a_l) \equiv 0 \mod a_i.\] 
    
    \noindent Let $\calQ$ be the cokernel of the double dual map 
    \[ \Omega_X \to \Omega_X^{[1]} \to \calQ \to 0.\]
    Then, the sheaf $Q_{ij}$ defined by $Q_{ij} = (\calQ|_{C_{ij}})/\tau$, where $\tau = Tors(\calQ|_{C_{ij}})$ is given by a sum of line bundles $Q_{ij} = \oplus_{m = 0}^{n-1} \calO(-b_m)$ where \[b_m = \frac{a_id_{i,m} + a_jd_{j,m} + (m+1)(a_k+ a_l)}{\frac{a_ia_j}{n+1}} > 0.\]  
    
    \noindent Furthermore, each $\calO(-b_m) = \calL_m|_{C_{ij}}$ for some $\bQ$-line bundle $\calL_m$ on $X$ that is Cartier in a neighborhood of $C_{ij}$.
\end{proposition}

\begin{proof}
We fix some notation for the proof: we will let $X = \bP(a_0,a_1,a_2,a_3)$ be a well-formed weighted projective space with coordinates $[x_0:x_1:x_2:x_3]$, and $C_{ij}$ a one-dimensional singular stratum corresponding to a nontrival common factor $n+1=\gcd(a_i,a_j)$. Without loss of generality, assume $i = 0$ and $j=1$, so $C_{ij} = C_{01}=\bP(a_0,a_1) $ is given by $ (x_2 = 0) \cap(x_3 = 0)$.  Then, $C_{01} \cong \bP^1_{[X:Y]}$ via the isomorphism  $X = x_0^{a_1/(n+1)}$ and $Y = x_1^{a_0/(n+1)}$.  Assume that $X$ has an $A_n$ singularity at the generic point of $C_{01}$, or equivalently, $a_2 +a_3 = 0 \mod n+1$. 

Let $U_0 = D(x_0) = \frac{1}{a_0}(a_1,a_2,a_3)$. Consider the exact sequence 
\[ \Omega_{U_0} \to \Omega_{U_0}^{[1]} \to \calQ|_{U_0} \to 0.\]
Because $U_0 = \bA^3/\mu_{a_0}$, $U_0 = \Spec(\bC[x_1,x_2,x_3])^{\mu_{a_0}}$, and $\Omega_{U_0}$ is therefore generated as a $\bC[x_1,x_2,x_3]^{\mu_{a_0}}$ by the differentials of the $\mu_{a_0}$-invariant functions, i.e. over any open set $V \subset U_0$, an element of $\Omega_{U_0}(V)$ is of the form $g\ df$, where $g,f \in \calO_{U_0}(V)$ are $\mu_{a_0}$-invariant functions.  By \cite{Kni73}, $\Omega_{U_0}^{[1]}$ is instead the $\mu_{a_0}$-invariant differentials, i.e. the invariants of $(\Omega_{\bA^3})^{\mu_{a_0}}$, where the action of $\mu_{a_0}$ on $dx_i$ is given by $\zeta_{a_0} \cdot dx_i = d(\zeta_{a_0}x_i) = \zeta_{a_0} d x_i$, where $\zeta_{a_{0}} \in \mu_{a_{0}}$. We will choose a generating set for $\calQ|_{U_i}$ containing the images of the invariant differentials $x_1^{d_{1,m}} (x_2x_3)^m x_2 dx_3$ for each $0 \le m \le n-1$, where $d_{1,m}$ is the minimal positive integer such that this is invariant.  These are independent along $C_{01}$ after inverting $x_i$ by Lemma \ref{lem:extofAn}, and along the (potentially) singular strata $C_{02}$ and $C_{03}$, the image of these differentials is supported only at $p_0 = [1:0:0:0]$.  Indeed, on $U_0 \cap U_3$, if $k$ is such that $x_3^k$ is an invariant function, $\frac{x_1^{d_{1,m}}(x_2x_3)^mx_2x_3}{kx_3^k}d(x_3^k) = x_1^{d_{1,m}}(x_2x_3)^mx_2dx_3$ is in the image of $\Omega_{U_0} \to \Omega_{U_0}^{[1]}$ so is equal to $0$ in $\calQ|_{U_0 \cap U_3}$.  A similar argument holds on $U_0 \cap U_2$.

    Since $Q_{01}$ is a torsion free coherent sheaf on $C_{01} \cong \bP^1$, it is a sum of line bundles.  At the generic point $\eta$ of $C_{01}$, by Lemma \ref{lem:extofAn}, this has rank $n$, and $(Q_{01})_\eta = \oplus_{m=0}^{n-1} \calO_{\eta} \cdot (x_2x_3)^m x_2 dx_3$.  Therefore, we must have $Q_{01} \cong \oplus_{m=0}^{n-1} \calO_{\bP^1}(\ell_m)$, and we compute $\ell_m$ by directly computing the transition functions.  

    Consider $U_0 = D(x_0) = \frac{1}{a_0}(a_1,a_2,a_3)$.  On this set, as in the set-up, there exist invariant differential forms $x_1^{d_{1,m}}(x_2x_3)^m x_2 dx_3$ where $d_{1,m}$ is the smallest positive integer such that $d_{1,m}a_1 + (m+1)(a_2+a_3) = 0 \mod a_0$.  Denoting by $U_1 = D(x_1)$, the restriction of this invariant form to $U_0 \cap U_1$ is a generator of the $m$th summand of $Q_{01}$, given above by $\calO(U_0 \cap U_1) \cdot (x_2x_3)^m x_2 dx_3$, and has image $0$ in all other summands.  Similarly, on $U_1$, a generator for the $m$th summand is given by $x_0^{d_{0,m}}(x_2x_3)^m x_2 dx_3$ where $d_{0,m}$ is the smallest positive integer such that $d_{0,m}a_0 + (m+1)(a_2+a_3) = 0 \mod a_1$.  To determine what line bundle the $m$th summand is, we compute the transition functions.  On the weighted projective space $\bP(a_0, a_1, a_2, a_3)$, the transition functions from $U_0$ to $U_1$ (i.e. the image of the restriction map $U_0 \to U_0 \cap U_1$ in the coordinates on $U_1$) are given by sending $x_1 \mapsto \frac{1}{x_0^{a_1/a_0}}$, $x_2 \mapsto \frac{x_2}{x_0^{a_2/a_0}}$, and $x_3 \mapsto \frac{x_3}{x_0^{a_3/a_0}}$.  Applying this to the invariant differentials generating $Q_{01}$ on $U_0$, we obtain the differential form 
    \[ x_1^{d_{1,m}}(x_2x_3)^m x_2 dx_3 \mapsto \frac{1}{x_0^{d_{1,m}a_1/a_0}} \frac{(x_2x_3)^m}{x_0^{m(a_2+a_3)/a_0}} \frac{x_2}{x_0^{a_2/a_0}} d\left( \frac{x_3}{x_0^{a_3/a_0}} \right).\] Note here that the fractional powers represent the transition functions in coordinates on the quotient $\bA^3/\mu_{a_1}$, see e.g. \cite[I.2.3]{Fletcher}.

    On $U_1$, for any power $p$ such that $x_0^p$ is an invariant function,  $\frac{1}{px_0^p}d(x_0^p) = \frac{dx_0}{x_0}$ is in the image of the map $\Omega_{U_1} \to \Omega_{U_1}^{[1]}$, so equal to $0$ in $\calQ$.  Therefore, on $U_0 \cap U_1$, any invariant differential $f dx_0 = 0 \in \calQ$, so we may simplify the above expression to  
    \[ x_1^{d_{1,m}}(x_2x_3)^m x_2 dx_3 \mapsto \frac{1}{x_0^{(d_{1,m}a_1+(m+1)(a_2+a_3))/a_0}} (x_2x_3)^m x_2dx_3,\] where $(d_{1,m}a_1+(m+1)(a_2+a_3))/a_0 \in \bZ$ by assumption.

    On $Q_{01} = \bP(a_0, a_1)$, the isomorphism $Q_{01} \cong \bP^1_{[X:Y]}$ is given by $X = x_0^{a_1/(n+1)}$ and $Y = x_1^{a_0/(n+1)}$.  On $\bA^1_Y = D(X)$, the $m$th summand of $Q_{01}$ is identified with $k[Y]$, where the isomorphism is given by
    \begin{align*}
        k[Y]  &\to  k[x_1^{a_0/(n+1)}] \cdot x_1^{d_{1,m}}(x_2x_3)^m x_2 dx_3 \\
        f(Y) &\mapsto f(x_1^{a_0/(n+1)})x_1^{d_{1,m}}(x_2x_3)^m x_2 dx_3.
    \end{align*}

    Similarly, on $\bA^1_X = D(Y)$, the $m$th summand of $Q_{01}$ is identified with $k[X]$, where the isomorphism is given by
    \begin{align*}
        k[X]  &\to  k[x_0^{a_1/(n+1)}] \cdot x_0^{d_{0,m}}(x_2x_3)^m x_2 dx_3 \\
        f(X) &\mapsto f(x_0^{a_1/(n+1)})x_0^{d_{0,m}}(x_2x_3)^m x_2 dx_3.
    \end{align*}

    Writing the transition function above as 
    \begin{align*}
    x_1^{d_{1,m}}(x_2x_3)^m x_2 dx_3 &\mapsto \frac{1}{x_0^{(d_{1,m}a_1+(m+1)(a_2+a_3))/a_0}} (x_2x_3)^m x_2dx_3 \\
    &= \frac{1}{(x_0^{a_1/(n+1)})^{(d_{0,m}a_0+d_{1,m}a_1+(m+1)(a_2+a_3))/(a_0a_1/(n+1))}} x_0^{d_{0,m}}(x_2x_3)^m x_2dx_3
    \end{align*}
    we see that the transition function $k[Y] \to k[X]$ is given by sending $1$ to $\frac{1}{Y^{b_m}}$ where $b_m = (d_{0,m}a_0+d_{1,m}a_1+(m+1)(a_2+a_3))/(a_0a_1/(n+1))$.  Therefore, the $m$th summand of $Q_{01}$ is precisely $\calO_{\bP^1}(-b_m)$.  

    By construction, $\calO_{\bP^1}(-b_m)$ is the restriction of 
    $\calO_{X}(d_{0,m}a_0+d_{1,m}a_1+(m+1)(a_2+a_3))$, which has degree divisible by $a_0$ and $a_1$, so is Cartier in a neighborhood of $C_{01}$. 
\end{proof}

\begin{lemma}\label{lem:descriptionofQ}
    Let $X = \bP(a_0, a_1,a_2,a_3)$ be a well-formed weighted projective threefold and let $C = \cup_{i,j}C_{ij}$ be one of the one-dimensional singular strata of $X$ such that $X$ has only $A_n$ singularities in codimension 2.  Denote by $f_{ij}: C_{ij} \to C$ the inclusion.
    
    Let $\calQ$ be the cokernel of the double dual map 
    \[ \Omega_X \to \Omega_X^{[1]} \to \calQ \to 0.\]
    Let $T = Tors(\calQ)$.  Then, $\calQ/T \cong \oplus_{i,j}  Q_{ij}$, and $T$ is a torsion sheaf supported at the codimension 3 singular points and the sheaf $Q_{ij}$ is defined as $Q_{ij} = (\calQ|_{C_{ij}})/\tau$, where $\tau = Tors(\calQ|_{C_{ij}})$ as above. 
\end{lemma}

\begin{proof}
Let $T = Tors(\calQ)$.  Consider the natural map 
\[ \calQ/T \to \oplus_{i,j} Q_{ij} .\]
Because $\calQ/T$ is torsion free, this map is injective, as its kernel can only be supported at the codimension 3 singular points of $X$, which are the (isolated) singular points of the curve $C$.  So, it suffices to prove it is surjective, which will follow from proving $\calQ \to \oplus_{i,j} Q_{ij}$ is surjective.  We do this by restricting the invariant differentials in $\calQ$ to the curves $C_{ij}$.  Let $U_i = D(x_i)$ and let (the image of) $x_i^{d_{i,m}}(x_kx_l)^m x_k dx_l \in \calQ \vert_{U_{i}}$ be a generator for the $m$th summand of $Q_{ij}$ as in the previous proof.  Then, the restriction of this invariant differential (thought of as an element of $\Omega_{U_i}^{[1]}$) to $C_{il}$ or $C_{ik}$ is torsion, as shown in the previous proof, so the restriction to $Q_{il}$ or $Q_{ik}$ is 0. Therefore, a generator of each summand of $Q_{ij}$ is in the image of $\Omega_X^{[1]}$, so in particular in the image of $\calQ$.  This holds for any $i,j$, so the map $\calQ \to \oplus_{i,j}Q_{ij}$ is surjective as desired.
\end{proof}

Finally, we combine the previous results into the desired theorem. 

\begin{theorem}\label{thm:Ext10}
    Let $X$ be a well-formed weighted projective threefold $X = \bP(a_0, a_1, a_2, a_3)$.  Assume that $X$ satisfies the assumptions in \ref{assumptions-on-wps}, i.e. has only singularities of type $A_n$ in codimension 2, and for each one dimensional singular strata $C_{ij}$ such that $X$ is singular at the generic point of $C_{ij}$, there exists a section $s_{ij} \in H^0(X, \mathcal{\omega}_X^{\vee})$ that does not generically vanish along $C_{ij}$.  Then, $H^1(X, \SExt^1(\Omega_X, \calO_X)) = 0$.
\end{theorem}

\begin{proof}
First, observe that $\SExt^1(\Omega_X, \calO_X)$ is supported on the singular locus of $X$, which is the union of the curves $C_{ij}$.  We will first compute the first cohomology of the restriction to these curves: 

\noindent \textbf{Claim 1.} Let $\calQ$ be the cokernel of the double dual map $\Omega_X \to \Omega_X^{[1]} \to \calQ$ as in Lemma \ref{lem:descriptionofQ}.  Then, \[\SExt^1(\Omega_X, \omega_X) \cong \SExt^2(\calQ/T, \omega_X) = \oplus_{i,j}(\SExt^2(Q_{ij}, \omega_X)).\]  

From the exact sequence 
\[ \Omega_X \to \Omega_X^{[1]} \to \calQ \to 0\]
which is injective in codimension 2 by the proof of Lemma \ref{lem:extofAn}, we see that the kernel $K$ is supported on the singular points $[1:0:0:0], [0:1:0:0], [0:0:1:0]$, and $[0:0:0:1]$ of $X$. From Ischebeck's lemma (see, e.g. \cite[Lemma 15.2]{Matsumura}), as $\dim \Supp(K) = 0$, $\SExt^0(K, \omega_X) = \SExt^1(K, \omega_X) = 0$. Applying this vanishing to the long exact sequence in $\SExt$ arising from the short exact sequence 
\[ 0 \to K \to \Omega_X \to \Omega_X/K \to 0 \] and applying $\SHom(-, \omega_X)$, we conclude that $\SExt^1(\Omega_X, \omega_X) = \SExt^1(\Omega_X/K, \omega_X)$.  Next, by Theorem \ref{thm:wps-dualizingprops} and \cite[Theorem 3.3.10]{BH98}, we have $\SExt^i(\Omega_X^{[1]}, \omega_X) = 0$ for $i > 0$.  This implies  $\SExt^1(\Omega_X, \omega_X) \cong \SExt^2(\calQ, \omega_X)$ by the usual long exact sequence.  

Now, let $C$ be the singular locus of $X$, so $C$ is a union of $\bP^1$s.  Viewing $Q$ as a sheaf on $C$, let $T = Tors(\calQ)$.  Then, $T$ is supported in dimension $0$ by Lemma \ref{lem:descriptionofQ}.  Again by Ischebeck's Lemma, we have $\SExt^1(T, \omega_X) = \SExt^2(T, \omega_X) = 0$, so applying the same argument as above to the short exact sequence 
\[ 0 \to T \to \calQ \to \calQ/T \to 0, \] we conclude $\SExt^2(\calQ, \omega_X) \cong \SExt^2(\calQ/T, \omega_X)$.  Together with Lemma \ref{lem:descriptionofQ}, this yields \[\SExt^1(\Omega_X, \omega_X) \cong \SExt^2(\calQ/T, \omega_X)=\oplus_{i,j}(\SExt^2(Q_{ij}, \omega_X)),\] and the claim is proved.  

Next, we restrict this sheaf to $C_{ij}$ to compute its cohomology.  

\noindent \textbf{Claim 2.} $H^1(C_{ij}, \SExt^2(Q_{ij}, \omega_X)\otimes\omega_X^{\vee}) = 0.$

By \Cref{prop:computingextoncurves}, $Q_{ij} = \oplus \calO(-b_m) $ where $b_m > 0$ is as in the statement and each line bundle $\calO(-b_m)$ is the restriction of a divisor $\calL_m$ on the threefold $X$ that is Cartier in a neighborhood of $C_{ij}$.  Therefore, we can identify 
\[ \SExt^2(Q_{ij}, \omega_X) = \SExt^2((\oplus \calL_m)\otimes\calO_{C_{ij}}, \omega_X) =\SExt^2(\calO_{C_{ij}}, \omega_X) \otimes (\oplus \calL_m)^{\vee} .\]

By Grothendieck duality, $\SExt^2(\calO_{C_{ij}}, \omega_X) \cong \omega_{C_{ij}} = \calO(-2)$, and $\calL_m^\vee|_{C_{ij}} = \calO(b_m)$, and therefore   
\[ \SExt^2(Q_{ij}, \omega_X) \cong \calO(-2) \otimes (\oplus \calO(b_m)) = \oplus_{m=0}^{n-1}\calO(-2+b_m).\]

Finally, by assumption, there exists a section of $\omega_X^\vee$ not vanishing at the generic point of $C_{ij}$, which implies that the map $\calO_X \to \omega_X^{\vee}$ corresponding to this section restricts to an injection $\calO_{C_{ij}} \to \omega_X^{\vee}|_{C_{ij}}$, and hence $\omega_X^{\vee}|_{C_{ij}} = \calO(f) \oplus T'$ for some $f \ge 0$ and $T'$ a torsion sheaf.  Therefore,
\[ \SExt^2(\calQ_{ij}, \omega_X) \otimes \omega_X^\vee = (\oplus \calO(-2+b_m)) \otimes (\calO(f) \oplus T') = (\oplus \calO(-2+b_m+f)) \oplus T''\]
where $b_m > 0$, $f \ge 0$, and $T''$ is torsion.  In particular, $-2+b_m + f \ge -1$, and therefore $H^1(C_{ij},\SExt^2(Q_{ij}, \omega_X) \otimes \omega_X^\vee) = 0$ and the proof of the claim is complete.

\noindent \textbf{Claim 3.} $H^1(X, \SExt^1(\Omega_X, \calO_X)) =0$.  

By Lemma \ref{lem:extdoubledual}, there is a natural map 
\[ \SExt^1(\Omega_X, \omega_X) \otimes \omega_X^{\vee} \to \SExt^1(\Omega_X, \calO_X) \] that is an isomorphism at the generic point of $C_{ij}$, so has kernel and cokernel supported only in dimension 0.  Therefore, 
\[ H^1(X,\SExt^1(\Omega_X, \omega_X) \otimes \omega_X^{\vee}) \cong H^1(X,\SExt^1(\Omega_X, \calO_X))\] so it suffices to show $H^1(X,\SExt^1(\Omega_X, \omega_X) \otimes \omega_X^{\vee}) = 0$. By Claim 1, 
\[\SExt^1(\Omega_X, \omega_X) \otimes \omega_X^{\vee} \cong \Bigl( \oplus_{i,j} \SExt^2(Q_{ij}, \omega_X) \Bigr) \otimes \omega_X^\vee,\] and by Claim 2,
\[H^1(C_{ij}, \SExt^2(Q_{ij}, \omega_X)\otimes\omega_X^{\vee}) = 0.\]  Therefore, $H^1(X,\SExt^1(\Omega_X, \omega_X) \otimes \omega_X^{\vee} ) = 0$, as desired.
\end{proof}

Finally, we give the proof of Theorem \ref{introthm:defofthreefolds}. 

\begin{proof}{(Proof of Theorem \ref{introthm:defofthreefolds}.)}
    This follows directly from Theorem \ref{thm:Ext10}, Corollary \ref{cor:wpssmoothing}, and Corollary \ref{cor:globaldefspaceislocal}.
\end{proof}

\section{Two additional examples}

The aim of this section is to prove the following theorem.

\begin{theorem}
    The two weighted projective threefolds $\bP(2,3,3,4)$ and $\bP(2,12,21,49)$ are each smoothable to a cubic threefold. Moreover, the weighted projective spaces $\bP(2,3,3,4,6_{n-4},18)$ and $\bP(2,12,21,42_{n-4},49,126)$ are smoothable to $\bP^{n}$ for any $n \ge 4$.
\end{theorem}

\begin{proof}
    We proceed in several parts.  First, assume that we have already shown that both $\bP(2,3,3,4)$ and $\bP(2,12,21,49)$ are smoothable to a cubic threefold $X \subset \bP^4$.  In the notation of Corollary \ref{cor:hypersurfaceDegeneration}, $n = 4$ and $d = 3$, so by Corollary \ref{cor:hypersurfaceDegeneration}, $\bP(2,3,3,4,18)$ and $\bP(2,21,21,49,126)$ are each smoothable to $\bP^4$. Now by Corollary \ref{cor:degenofPn+1}, this implies 
    that
    $\bP(2,3,3,4,6_{n-4},18)$ and $\bP(2,12,21,42_{n-4},49,126)$ are smoothable to $\bP^{n}$ for any $n \ge 4$.

    Therefore, to prove the theorem, it suffices to show $\bP(2,3,3,4)$ and $\bP(2,12,21,49)$ are smoothable to a cubic threefold as claimed.  

    For $\bP(2,3,3,4)$, denote its coordinates by $[x_0:x_1:x_2:x_3]$. As it has only $A_1$ and $A_2$ singularities in codimension 2 and $-K_{\bP(2,3,3,4)}$ is very ample, by Theorem \ref{thm:Ext10} and Corollary \ref{cor:wpssmoothing}, it suffices to verify that there is a local smoothing at each of the following points that agree on the overlaps:
    \begin{center}
    $[1:0:0:0], [0:1:0:0], [0:0:1:0], [0:0:0:1]$.
    \end{center}
    The first three of these points are $A_n \times \bA^1$ singularities for $n = 1,2$, so they are smoothable, and the last point is a $\frac{1}{4}(1,1,2)$ singularity, which is smoothable by either Example \ref{ex:VeroneseEmbeddingTrick} or Example \ref{ex:P(2,12,21,49)-ps}.  It is straightforward to show these smoothings can be chosen to agree on the overlaps of the sets $U_i$.  Therefore, by Theorem \ref{thm:Ext10} and Corollary \ref{cor:wpssmoothing}, this is smoothable.  One could also write the smoothing explicitly: consider the `degree 6 embedding' 
    \[ \bP(2,3,3,4) \hookrightarrow \bP(1,1,1,1,1,2)\] given by 
    \[ [x_0:x_1:x_2:x_3] \mapsto [x_1^2:x_1x_2:x_2^2:x_0^3:x_0x_3:x_3^3].\]
    Denoting the coordinates on $\bP(1,1,1,1,1,2)$ by $[y_i]$, the image is defined by $y_0 y_2 - y_1^2$ and $y_3y_5 - y_4^3$.  Perturbing these equations over $\bA^1_t$ by $y_0 y_2 - y_1^2 = ty_5$ and $y_3y_5 - y_4^3 = tf_3$ for a generic degree 3 equation $f_3=f_3(y_0,y_1,y_2,y_3,y_4)$, the central fiber $t = 0$ is the image of $\bP(2,3,3,4)$ and the general fiber $t \ne 0$ is 
    \[ (y_3(y_0 y_2 - y_1^2) - ty_4^3 = t^2f_3(y_0,y_1,y_2,y_3,y_4)) \subset \bP(1,1,1,1,1) = \bP^4. \] For generic $t$, this is a smooth cubic threefold. 

    The proof that $\bP(2,12,21,49)$ is smoothable to a cubic will occupy the rest of this section.
\end{proof}

Consider the weighted projective space $\mathbb{P}(2, 12, 21, 49)$. First, we describe the singularities of this space. 

At the point $p_0 = [1:0:0:0]$, there is a $\frac{1}{2}(0,1,1)$ singularity; at 
$p_1 = [0:1:0:0]$, there is a $\frac{1}{12}(1,2,9)$ singularity; at the point $p_2 = [0:0:1:0]$, there is a $\frac{1}{21}(2,12,7)$ singularity; and at the point $p_3 = [0:0:0:1]$, there is a $\frac{1}{49}(2,12,21)$ singularity.  There are curves of singularities joining these points: along the curve $C_{01}$ joining the points $p_0$ and $p_1$, the threefold has $\frac{1}{2}(1,1) \times \bA^1$ singularities.  Along the curve $C_{12}$ joining the points $p_1$ and $p_2$, the threefold has $\frac{1}{3}(1,2) \times \bA^1$ singularities.  Finally, along the curve $C_{23}$ joining $p_2$ and $p_3$, the threefold has $\frac{1}{7}(1,6) \times \bA^1$ singularities. 

Now, we aim to show that $\bP(2,12,21,49)$ is smoothable to a cubic, through a series of lemmas.  In what follows, the notation $\bP(a_0,\dots,a_n)_{[x_i]}$ indicates that the coordinates on the weighted projective space are given by $[x_0:\dots:x_n]$. 

\begin{lemma}\label{lem:psofP2122149}
    The weighted projective space $\bP(2,12,21,49)$ admits a partial smoothing to a degree 10 
    hypersurface $X_{10} \subset \bP(1,1,3,6,5)_{[w_i]}$ of the form $w_0w_2w_4 = w_1^{10} + aw_1^7w_2 + bw_1w_2^3 +w_3^2$ for some $a,b \in \bC$.
\end{lemma}

\begin{proof}
    Consider the embedding $\bP(2,12,21,49)_{[x_i]} \to \bP(2,2,3,7,12)_{[y_i]}$ given by 
    \[ [x_0:x_1:x_2:x_3] \mapsto [x_0^7:x_0x_1:x_2:x_3:x_1^7].\]
    The image is defined by the degree 14 equation $y_0y_4 = y_1^7$.  Let $X_{14}(t,s)$ be the degree 14 weighted hypersurface $y_0y_4 = y_1^7 + t y_3^2 + s y_2^4y_1$, which admits an isotrivial specialization to $\bP(2,12,21,29)$ sending $t, s \to 0$.

    Now, consider the embedding $\bP(2,2,3,7,12)_{[y_i]} \to \bP(1,1,3,5,6,7)_{[w_i]}$ given by 
    \[[y_0:y_1:y_2:y_3:y_4] \mapsto [y_0:y_1:y_2^2:y_2y_3:y_4:y_3^2] \]
    with image defined by the degree 10 equation $w_2w_5 = w_3^2$.  Let $Y_{7}(t,s)$ be the weighted hypersurface given by $w_0w_4 = w_1^7 + tw_5 + s w_2^2w_1$, so the image of $X_{14}(t,s)$ is given by the intersection of $Y_{7}(t,s)$ and $w_2w_5 = w_3^2$.  For $t \ne 0$, this is isomorphic to the degree 10 weighted hypersurface $X_{10}(t,s)$ given by $w_2(w_0w_4 - w_1^7 - sw_2^2 w_1) = tw_3^2$ in the weighted projective space $\bP(1,1,3,5,6)_{[w_i]}$.  For any $t \ne 0$, this is isomorphic to the same equation with $t = 1$.  Perturbing this equation by $w_1^{10}$, we see that $X_{10}(t,s)$ admits a smoothing to a hypersurface of the form $X_{10}$ as claimed.

    We conclude $\bP(2,12,21,49)$ admits a partial smoothing to $X_{10}$ by the following lemma.
\end{proof}

\begin{lemma}\label{lem:xtoytoz}
    Let $X,Y,Z$ be Fano varieties.  Suppose $X$ admits a $\bQ$-Gorenstein degeneration to $Y$ and $Y$ admits an isotrivial $\bQ$-Gorenstein degeneration to $Z$.  Then, $X$ admits a $\bQ$-Gorenstein degeneration to $Z$.
\end{lemma}

\begin{proof}
    Choose a sufficiently divisible $m \gg 0$ such that $\omega_{X}^{[m]}$, $\omega_{Y}^{[m]}$, and $\omega_{Z}^{[m]}$ are very ample.  By flatness of $\bQ$-Gorenstein degenerations, using these very ample line bundles, $X$, $Y$, and $Z$ all admit embeddings into the same projective space $\bP^N$ with the same Hilbert polynomial.  Therefore, they correspond to points $[X], [Y], [Z]$ of $\Hilb(\bP^N)$.  By assumption, $[Y]$ is a specialization of $[X]$ and $[Z]$ is a specialization of $[Y]$.  As $\Hilb(\bP^N)$ is projective, this implies $[Z]$ is a specialization of $[X]$.  
\end{proof}

To prove $\bP(2,12,21,49)$ is smoothable, again using the previous lemma, it suffices to show $X_{10}$ is smoothable. This will be accomplished by the next few statements.  First, we prove an auxillary result that a hypersurface of a similar form admits a smoothing to a quadric threefold.   

\begin{lemma}\label{lem:deformationtoquadric}
    Let $Y_{10}$ be a degree 10 hypersurface in $\bP(1,1,3,5,9)$ with equation of the form $(z_0z_4 = z_1^{10} + az_1^7z_2 + bz_1z_2^3+ z_3^2)$ for some $a,b \in \bC$.  Then, $Y_{10}$ admits a smoothing to a quadric threefold in $\bP^4$. 
\end{lemma}

\begin{proof}
    By \cite{HP10}, the weighted projective space $\bP(1,2,9)$ and its partial smoothings admit a smoothing to $\bP^1 \times \bP^1$.  One partial smoothing is given by the degree 10 equation $Z_{10} = (z_0z_3 = z_1^{10} + z_2^2) \subset \bP(1,1,5,9)_{[z_i]}$.  This smoothing comes from the degree 2 embedding of $\bP(1,2,9)_{[x_i]} \hookrightarrow\bP(1,1,5,9)_{[z_i]}$ given by $[x_0:x_1:x_2] \mapsto [x_0^2: x_1:x_0x_2: x_2^2]$.  The image is defined by $z_0z_3 = z_2^{2}$ and a partial smoothing is given by $z_0 z_2 = tz_1^{10} + z_2^2$.  For any $t \ne 0$, this is isomorphic to $Z_{10}$. 

    We note that
    in this deformation, any divisor in the linear system $\calO_{Z_{10}}(3k)$ may be deformed to a divisor in the linear system $\calO_{\bP^1 \times \bP^1}(k,k)$. Indeed, recall that the smoothing of $Z_{10}$ to $\bP^1 \times \bP^1$ is $\bQ$-Gorenstein, and by adjunction, the anticanonical divisor on $Z_{10}$ is $\calO_{\bP(1,1,5,9)}(6)|_{Z_{10}} = \calO_{Z_{10}}(6)$.  Because the anticanonical divisor on $\bP^1 \times \bP^1$ is $\calO_{\bP^1 \times \bP^1}(2,2)$, this implies that divisors in the linear system $\calO_{Z_{10}}(6)$ deform to divisors in the linear system $\calO_{\bP^1 \times \bP^1}(2,2)$.  As $Z_{10}$ has Picard rank one, divisors are determined by their degree, and this implies divisors in the linear system $\calO_{Z_{10}}(3k)$ deform to divisors in the linear system $\calO_{\bP^1 \times \bP^1}(k,k)$.  (Note that all sections of these linear systems deform because $H^1(Z_{10}, \calO_{Z_{10}}) = 0$.)
    
    As $\bP^1 \times \bP^1$ is a quadric hypersurface in $\bP^3$, by Theorem \ref{thm:conesanddivisors}, $C(\bP^1 \times \bP^1, \calO(1))$ admits a degeneration to the cone $C_p(Z_{10}, \calO(3))$, which is given by 
    \[
    Y_{10} = (z_0z_4 = z_1^{10} + z_3^2)\subset \bP(1,1,3,5,9)_{[y_i]}.
    \] As the cone $C(\bP^1 \times \bP^1, \calO(1))$ is simply a singular quadric threefold, by Lemma \ref{lem:xtoytoz}, we conclude that a smooth quadric threefold admits a degeneration to $Y_{10}$.  

    Unfortunately, $Y_{10}$ is not sufficiently generic: the previous argument does not imply a more generic equation $Y_{10}' = (z_0z_4 = z_1^{10} + az_1^7z_2 + bz_1z_2^3+ z_3^2)\subset \bP(1,1,3,5,9)_{[y_i]}$ admits a smoothing to the quadric threefold.  To verify the existence of a smoothing, we first compute $H^1(Y_{10}', \SExt^1(\Omega_{Y_{10}'}, \calO_{Y_{10}'}))$ with the same methods as in Section \ref{sec:deformations}.  Observe that $Y_{10}'$ is quasismooth and only singular along the curve $C = \bP(3,9)_{[z_2:z_4]}\cong \bP^1_{[z_2^3:z_4]}$ given by $z_0 = z_1 = z_3 = 0$ with a $\frac{1}{3}(1,2,0)$ singularity at all points where $z_2 \ne 0$ and a $\frac{1}{9}(1,3,5)$ singularity at the point $[0:0:0:0:1] \in Y_{10}' \subset \bP(1,1,3,5,9)$.  By adjunction, the canonical divisor of $Y_{10}'$ has (weighted) degree 9, which is Cartier at all points of $Y_{10}'$, so $Y_{10}'$ is Gorenstein.  

    Where $z_2z_4 \ne 0$, the equation for $Y_{10}'$ realizes $z_0$ in terms of the remaining coordinates, and along $C \cap(z_2z_4 \ne 0)$, by Lemma \ref{lem:extofAn}, we have $z_1 dz_3$ and $(z_1z_3) z_1 dz_3$ are generators for the sheaf $\SExt^1(\Omega_{Y_{10}'}, \calO_{Y_{10}'})|_{(z_2z_4 \ne 0)}$. Furthermore, both $z_1 dz_3$ and $(z_1z_3) z_1 dz_3$ are invariant differentials on the affine open set $(z_2 \ne 0)$.  Along the open set $(z_4 \ne 0)$, these can be realized as the restriction of the invariant differentials $z_2 z_1 dz_3$ and $z_2^2(z_1z_3) z_1 dz_3$.  As in Proposition \ref{prop:computingextoncurves}, computing the transition functions between the charts $(z_2 \ne 0)$ and $(z_4 \ne 0)$, we find that the cokernel $\calQ$ of the double dual map 
    \[ \Omega_{Y_{10}'} \to \Omega_{Y_{10}'}^{[1]} \to \calQ \to 0\] is $\calQ|_C/(Tors(\calQ|_C)) = \calO_{\bP^1}(-1) \oplus \calO_{\bP^1}(-2)$.  By direct computation, on $C$, $\omega_{Y_{10}}|_C \cong \calO_{\bP^1}(-1)$.  Now, following the claims in the proof of Theorem \ref{thm:Ext10}, from Claim 1 we conclude $\SExt^1(\Omega_{Y_{10}'}, \omega_{Y_{10}'})\cong \SExt^2(\calQ|_C/(Tors(\calQ|_C)), \omega_{Y_{10}'}))$, and from the argument in Claim 2, using that $Y_{10}'$ is Gorenstein, we conclude 
    \[\SExt^1(\Omega_{Y_{10}'}, \calO_{Y_{10}'}) \cong \SExt^1(\Omega_{Y_{10}'}, \omega_{Y_{10}'}) \otimes \omega_{Y_{10}'}^\vee = \calO_{\bP^1}(-2+1+1) \oplus \calO_{\bP^1}(-2+2+1) = \calO_{\bP^1} \oplus \calO_{\bP^1}(1).   \]
    This implies that $H^1(\SExt^1(\Omega_{Y_{10}'}, \calO_{Y_{10}'})) = 0$.

    Because $Y_{10}'$ is a Gorenstein Fano hypersurface in weighted projective space, we also have $H^2(Y_{10}', T_{Y_{10}'}) = 0$.  Indeed, because $Y_{10}'$ is Gorenstein, by \cite[Proposition 4.3]{GKP}, we have $H^1(Y_{10}', \Omega_{Y_{10}'}^{[1]} \otimes \omega_{Y_{10}'}) = 0$, which is Serre Dual to $\Ext^2(\Omega_{Y_{10}'}^{[1]} \otimes \omega_{Y_{10}'}, \omega_{Y_{10}'})$.  However, $Y_{10}'$ has only cyclic quotient singularities, so $\Omega_{Y_{10}'}^{[1]}$ is Cohen-Macaulay (this is a local statement, so follows from Theorem \ref{thm:wps-dualizingprops}(4)) and therefore $0 = \Ext^2(\Omega_{Y_{10}'}^{[1]} \otimes \omega_{Y_{10}'}, \omega_{Y_{10}'}) = H^2(Y_{10}', T_{Y_{10}'})$.  Thus, Sequence \ref{eqn:localtoglobalExt} implies that there is a surjection 
    \[ \Ext^1(\Omega_{Y_{10}'}, \calO_{Y_{10}'}) \to H^0(Y_{10}', \SExt^1(\Omega_{Y_{10}'}, \calO_{Y_{10}'}))\] and an injection 
    \[ \Ext^2(\Omega_{Y_{10}'}, \calO_{Y_{10}'}) \hookrightarrow H^0(Y_{10}', \SExt^2(\Omega_{Y_{10}'}, \calO_{Y_{10}'})).\]

    Because $Y_{10}'$ has only $\frac{1}{3}(1,2,0)$ singularities away from the point $[0:0:0:0:1]$, by the proof of Lemma \ref{lem:AnExt2} the obstruction $\SExt^2(\Omega_{Y_{10}'}, \calO_{Y_{10}'})$ is supported only at the $\frac{1}{9}(1,3,5)$ singular point $[0:0:0:0:1]$.  Letting $U_2 = (z_2 \ne 0)$ and $U_4 = (z_4 \ne 0)$, this says given any deformation of $U_2$ and $U_4$ that agree on their intersection, there is a corresponding deformation of $Y_{10}'$.  Now, we know the singularity $\frac{1}{9}(1,3,5)$ in the chart $U_4$ is smoothable because $Y_{10}$ was smoothable (and $Y_{10}$ and $Y_{10}'$ are isomorphic on the chart $U_4$).  The restriction of this deformation to $U_2 \cap U_4$ is some element of $H^0(U_2 \cap U_4, \SExt^1(\Omega_{U_2 \cap U_4}, \calO_{U_2 \cap U_4}))$  On the chart $U_2$, $Y_{10}'$ has only singularities of the form $A_2 \times \bA^1$ which are smoothable and unobstructed.  Choosing a smoothing of $U_2$ that restricts to the same element of $H^0(U_2 \cap U_4, \SExt^1(\Omega_{U_2 \cap U_4}, \calO_{U_2 \cap U_4}))$, which exists because $H^1(\SExt^1(\Omega_{Y_{10}'}, \calO_{Y_{10}'})) = 0$ and $U_2$ is smoothable and unobstructed, we conclude there is a compatible smoothing of $U_2$.  Therefore, there exists a smoothing of $Y_{10'}$.  

    Finally, this must give a smoothing of $Y_{10'}$ to the quadric, because by construction the smoothing is $\bQ$-Gorenstein and by computation $Y_{10'}$ has anticanonical volume $54$, which is preserved by the smoothing.  This implies the smoothing $V$ of $Y_{10'}$ is a Fano threefold with anticanonical volume $54$.  Also, observe that there are no obstructions to extending the divisor $\calO_{Y_{10}'}(3)$ in this smoothing: by Theorem 5.24 of \cite{TV}, obstructions to deforming a sheaf $\calF$ are contained in $\SExt^{2}(\mathcal{F}, \mathcal{F})$.  For $\calF = \calO_{Y_{10}'}(3)$, this vanishes at all points other than the $\frac{1}{9}(1,3,5)$ singular point $[0:0:0:0:1]$ because $\calF$ is Cartier away from that point.  At that point, our deformation coincides with the smoothing of $Y_{10}$ to the quadric threefold, and the divisor $\calF$ extends to this deformation (as $\calO_{Y_{10}}(3)$ is the limit of $\calO(1)$ on the quadric).  Therefore, the obstructions also vanish at this point. 

    Because $\calO_{Y_{10}'}(3 \cdot 3) = \calO_{Y_{10}'}(-K_{Y_{10}'})$, we conclude the Cartier index of $V$ is at least 3, and therefore $V$ must be a quadric threefold. 
\end{proof}

Now, we will finally show that $X_{10}$ admits a smoothing to a cubic threefold.

\begin{proposition}\label{prop:smoothingtocubic}
    Let $X_{10}$ be a degree 10 hypersurface in $\bP(1,1,3,6,5)$.  Then, $X_{10}$ admits a smoothing to a cubic threefold in $\bP^4$. 
\end{proposition}

\begin{proof}
    By Lemma \ref{lem:deformationtoquadric}, for $Y_{10}$ a degree 10 hypersurface in $\bP(1,1,3,5,9)$ with equation of the form $(z_0z_4 = z_1^{10} + az_1^7z_2 + bz_1z_2^3+ z_3^2)$ for some $a,b \in \bC$, $Y_{10}$ admits a smoothing to a quadric threefold $V$ in $\bP^4$.  Furthermore, in this smoothing, the $\bQ$-Cartier divisor $\calO_{Y_{10}}(3)$ deforms to the divisor $\calO_V(1)$ on the quadric. 
    
    As $V$ is a quadric hypersurface in $\bP^4$, by Theorem \ref{thm:degree-d-hypersurfaces}, $\bP^4$ admits a degeneration to the cone $C_p(V, \calO_V(2))$.  By Theorem \ref{thm:conesanddivisors}, $C_p(V, \calO_V(2))$ admits a degeneration to the cone $C_p(Y_{10}, \calO_{Y_{10}}(6))$.  By Lemma \ref{lem:xtoytoz}, we conclude $\bP^4$ admits a degeneration to $C(Y_{10}, \calO(6))$, which is given by $Z_{10} = (z_0z_4 = z_1^{10} + az_1^7z_2 + bz_1z_2^3+ z_3^2) \subset \bP(1,1,3,5,9,6)_{[z_i]}$. By Theorem \ref{thm:conesanddivisors}, any divisor in the linear system $\calO_{Z_{10}}(3k)$ deforms to a divisor of degree $k$ on $\bP^4$.  Then, for any $V$ a degree 9 hypersurface on $Z_{10}$, $V$ may be deformed to a cubic threefold on $\bP^4$.  But, the degree 9 hypersurface $z_4 = z_2z_5$ on $Z_{10}$ is just the threefold $X_{10}$.  Therefore, $X_{10}$ may be deformed to a cubic threefold. 
\end{proof}

There is an alternative geometric proof that `nodal' elements $X_{10}^n$ given by an equation of the form
\begin{center}
$z_0(z_2z_4 + g_9(z_0,z_1)) = z_3^2 + z_1^{10}$
\end{center}
admit smoothings to nodal cubic threefolds. The idea of the proof is to mimic the standard rationality construction of a nodal cubic threefold.  Let $Y \subset \bP^4_{[x_i]}$ be a nodal cubic threefold, given in coordinates by $x_4 f_2(x_0,x_1,x_2,x_3) = f_3(x_0,x_1,x_2,x_3)$ where $f_i$ is homogeneous of degree $i$.  Projecting from the node yields a degree one rational map $Y \dashrightarrow \bP^3$.  This map is resolved by blowing up the node in a morphism $Z \to Y$, yielding a morphism $Z \to \bP^3$ that is the blow-up of the $(2,3)$ complete intersection curve given by $f_2 = f_3 = 0$ in $\bP^3$.  This $(2,3)$ complete intersection curve is a canonical genus 4 curve.  We will show that $X_{10}$ can be constructed in the same way where $\bP^3$ and the canonical genus 4 curve are replaced by a klt degeneration of $\bP^3$ containing a hyperelliptic genus 4 curve.  Starting from $\bP^3$ and its degeneration and blowing up the family of curves will produce the degeneration of the cubic threefold to $X_{10}^n$. We include the argument below.

First, we realize hyperelliptic genus 4 curves in a klt degeneration of $\bP^3$.  

\begin{lemma}\label{lem:hyperellipticgenus4}
    Let $C$ be a hyperelliptic genus 4 curve.  Then, $C$ admits an embedding into a degree $10$ hypersurface $W_{10} \subset \bP(1,1,5,6,9)$ as a $(6,9)$ weighted complete intersection.  Furthermore, $W_{10}$ is smoothable to $\bP^3$ and, in this smoothing, the $(6,9)$ complete intersection deforms to a $(2,3)$ complete intersection.
\end{lemma}

\begin{proof}
    Every hyperelliptic curve $C$ of genus $4$ admits an embedding into $\bP(1,1,5)_{[x_0:x_1:x_2]}$ given by a degree $10$ equation of the form $x_2^2 = f_{10}(x_0,x_1)$ (see, e.g. \cite[Theorem 3.3]{ADLW}).  Up to change of coordinates, we may assume the polynomial $f_{10}(x,y)$ is given by $f_{10}(x_0,x_1) = x_0g_9(x_0,x_1) - x_1^{10}$, so the curve $C$ is given by $x_0 g_9 = x_1^{10} + x_2^2$.  

    Let $Z_{10} = (x_0x_3 = x_1^{10} + x_2^2) \subset \bP(1,1,5,9)_{[x_i]}$ be the partial smoothing of $\bP(1,2,9)$ as in Lemma \ref{lem:deformationtoquadric}, which is smoothable to $\bP^1 \times \bP^1$.  By construction, $C$ is the curve on $Z_{10}$ given by $x_3 = g_9$.  By Theorem \ref{thm:conesanddivisors} and the same argument in Proposition \ref{prop:smoothingtocubic}, if $W_{10} = C_p(Z_{10}, \calO(6))$ which is given by $W_{10} = (y_0 y_4 = y_1^{10} + y_2^2) \subset \bP(1,1,5,6,9)$, we see that $W_{10}$ is smoothable to $\bP^3$, and $Z_{10} \subset W_{10}$ as a divisor of degree 6.
    
    Furthermore, the deformation of the degree 10 equation to $\bP^1 \times \bP^1$, the line bundle $\calO_{Z_{10}}(3)$ deforms to $\calO_{\bP^1 \times \bP^1}(1,1)$.  There are no obstructions to deforming the Cartier divisor $x_3 = g_{9}(x_0,x_1)$ on $Z_{10}$ to $\bP^1 \times \bP^1$ and therefore this Cartier divisor may be deformed to a generic $(3,3)$ curve on $\bP^1 \times \bP^1$.  By Theorem \ref{thm:conesanddivisors}, the cones over these divisors also deform, and therefore, we obtain $C$ as a $(6,9)$ complete intersection in $W_{10}$ which deforms to a $(2,3)$ complete intersection in the smoothing of $W_{10}$ to $\bP^3$.
\end{proof}

Now, we prove that we may project away from a point $X_{10} \dashrightarrow W_{10}$.  We define the projection and collect several facts in the following lemma, the proof of which is straightforward computation.

\begin{lemma}
Let $X_{10}$ be a degree 10 hypersurface on $\bP(1,1,3,5,6)_{[z_i]}$ of the form
\[  z_1^{10} + z_3^2 = z_0(z_2z_4 + g_9(z_0,z_1,z_4))\] for some choice of $g_9$ that only depends on $z_0,z_1,z_4$.  

Consider the rational map $\bP(1,1,3,5,6)_{[z_i]} \dashrightarrow \bP(1,1,5,6,9)_{[y_i]}$ given by \[ [z_0:z_1:z_2:z_3:z_4] \mapsto [z_0:z_1:z_3:z_4:z_2z_4]. \]  

This is undefined only where $z_0 = z_1 = z_3 = z_4 = 0$, which is only the point $[0:0:1:0:0]$.  The image of $X_{10}$ is then given by 
\[ y_1^{10} + y_2^2 = y_0y_4 +y_0g_9(y_0,y_1,y_3)\]
so the image of $X_{10}$ is the threefold $W_{10}$ constructed in Lemma \ref{lem:hyperellipticgenus4}.

The locus on which this map has positive dimensional fibers is given by
the equations
$z_4 = 0$ and $z_1^{10} + z_3^2 = z_0g_9(z_0,z_1))$ in $\bP(1,1,3,5,6)_{[z_i]}$, which is a cone over the hyperelliptic curve $C_0$ defined by $z_1^{10} + z_3^2 = z_0g_9(z_0,z_1))$ in $\bP(1,1,5)_{[z_0:z_1:z_3]}$.  The image of these fibers is the locus $y_3 = y_4 = 0$ in $Y_{10}$, which is just the image of the hyperelliptic curve $C_0$. 

This rational map may be resolved by the $(1,1,5,9)$ weighted blow-up in coordinates 
\begin{center}
$(z_0,z_1,z_3,z_4) = (0,0,0,0)$ on $X_{10}$,
\end{center}
followed by the contraction of (the strict transform of) the rulings of the cone over the $C_0$.
\end{lemma}

We can perform this construction in a family starting from the degeneration of $\bP^3$ to $W_{10}$.  After blowing up a family of canonical genus 4 curves degenerating to a smooth hyperelliptic curve on $W_{10}$, there is a contraction of a family of surfaces realizing the partial smoothing of a nodal hypersurface $X_{10}^n$ to a nodal cubic threefold. 

\section{Applications}
\label{sec:FurtherApplications}
In this section, we explain some applications of the results from this paper, as well as some further possible directions and open questions.

\subsection{Terminal and canonical degenerations of $\bP^n$}

While the results in this paper are mostly focused on weighted projective spaces and klt degenerations of $\bP^n$, in this section we collect some results on terminal and canonical degenerations arising as both cones and weighted projective spaces. 

First, we have a complement to the classification of terminal degenerations of $\bP^3$ in higher dimensions.  

\begin{theorem}\label{thm:terminalgorensteindegens}
    For any $n \ge 4$, there exist terminal degenerations of $\bP^n$ with isolated singularities.  Furthermore, suppose $n+1$ is composite and $n > 4$.  Then, there exist terminal Gorenstein degenerations of $\bP^n$ with isolated singularities. 
\end{theorem}

The result will follow easily from the following Corollary of Proposition \ref{prop:singsofcones}. 

\begin{corollary}\label{cor:singsofconesoverhypersurfaces}
    Let $X_d \subset \bP^n$ be a smooth degree $d$ Fano hypersurface, $d \le n$.  Then, $C_p(X_d, \calO(d))$ is:
        \begin{enumerate}
            \item terminal if and only if $n+1 >2d$;
            \item canonical if and only if $n+1  \ge 2d$;
            \item Gorenstein if and only if $n +1 = md$ where $m \in \bZ$.
        \end{enumerate}
\end{corollary}

\begin{proof}
    This is a direct application of Proposition \ref{prop:singsofcones}, where $L = \calO(d)$ and $r = -\frac{n+1-d}{d}$.
\end{proof}

Now, we prove Theorem \ref{thm:terminalgorensteindegens}.

\begin{proof}
    By Theorem \ref{thm:degree-d-hypersurfaces}, $C_p(X_d, \calO(d))$ is a $\bQ$-Gorenstein degeneration of $\bP^n$ with isolated singularities.  By the previous Corollary, for any $d$ such that $1<d < \frac{n+1}{2}$ (for example, we may take $d = 2$ for any $n  \ge 4$), $C_p(X_d, \calO(d))$ has an isolated terminal singularity. 

    If additionally $n+1$ is composite, then by assumption there exist positive integers $a,b > 1$ such that $n+ 1 = ab$, and if $n > 4$, at least one of $a,b > 2$.  Without loss of generality, assume $a > 2$ and let $d = b$.  Then, $n+1 > 2d$ and $n+1 = ad$ where $a \in \mathbb{Z}$, so (1) and (3) of \Cref{cor:singsofconesoverhypersurfaces} imply that $C_p(X_d, \calO(d))$ has an isolated terminal Gorenstein singularity.
\end{proof}

We also mention canonical singularities. 

\begin{theorem}
    For any $n \ge 3$, there exist canonical degenerations of $\bP^n$ with isolated singularities.  For any odd number $n \ge 3$, there exist canonical, non-terminal Gorenstein degenerations of $\bP^n$ with isolated singularities.  
\end{theorem}

\begin{proof}
    The first statement follows from applying \Cref{thm:degree-d-hypersurfaces}, and then (2) of Corollary \ref{cor:singsofconesoverhypersurfaces} for any $1 < d \le \frac{n+1}{2}$; the resulting cone $C_p(X_d, \calO(d))$ has an isolated canonical singularity.  The second follows from the same construction letting $d = \frac{n+1}{2}$.  
\end{proof}

In addition to cones, the weighted projective spaces constructed in this paper also give examples of weighted degenerations of $\bP^n$ with non-isolated canonical singularities (unlike the previous examples of isolated singularities, these have singularities in codimension 2).  

\begin{theorem}
    Let $X$ be one of the weighted projective spaces 
        \begin{enumerate}
            \item $\bP(a^2,b^2,c^2)$, where $a^2 + b^2 + c^2 = 3abc$
            \item $\bP(a^2,b^2, 2c^2, 4abc)$, where $a^2 + b^2 + 2c^2 = 4abc$
            \item $\bP(2,3,3,4,18)$
            \item $\bP(2,12,21, 49, 126)$
        \end{enumerate}

    Let $X_n$ be the associated weighted projective space that admits a smoothing to $\bP^n$, which are (respectively) 
        \begin{enumerate}
            \item $\bP(a^2,b^2,c^2,(abc)_{n-2})$, where $a^2 + b^2 + c^2 = 3abc$
            \item $\bP(a^2,b^2, 2c^2, 4abc, (2abc)_{n-3})$, where $a^2 + b^2 + 2c^2 = 4abc$
            \item $\bP(2,3,3,4,6_{n-4},18)$
            \item $\bP(2,12,21,42_{n-4}, 49, 126)$
        \end{enumerate}
        
    For each choice of $X$, there exist infinitely many $n > 0$ such that $X_n$ is a degeneration of $\bP^n$ with canonical Gorenstein singularities. 
\end{theorem}

\begin{proof}
    We include the proof for $X = \bP(a^2,b^2,c^2)$ as the other cases are similar.  We wish to show there exist infinitely many $n > 0$ such that $X_n = \bP(a^2,b^2,c^2,(abc)_{n-2})$ has canonical Gorenstein singularities.  By the Reid-Tai criterion (see, for example  \cite{R80} and \cite{T82}), a cyclic quotient singularity $\frac{1}{b}(a_1, \dots, a_n)$ has canonical Gorenstein singularities if $a_1 + \dots + a_n \equiv 0 \pmod b$. 

    The singularities of $X_n$ are of the form
    \begin{align*}
    &\frac{1}{a^2}(b^2, c^2, (abc)_{n-2}) \\
    &\frac{1}{b^2}(a^2, c^2, (abc)_{n-2}) \\
    &\frac{1}{c^2}(a^2, b^2, (abc)_{n-2}); \text{ and } \\
    &\frac{1}{abc}(a^2, b^2, c^2, 0_{n-3}) 
    \end{align*}
    
    The last singularity is canonical Gorenstein by the assumption that $a^2 + b^2 + c^2 = 3abc$ and by the Reid-Tai criterion. So it suffices to find $n > 0$ such that the first three are canonical Gorenstein. By the Reid-Tai criterion, $\frac{1}{a^2}(b^2, c^2, (abc)_{n-2})$ is canonical Gorenstein if $b^2 + c^2 + (n-2)(abc) \equiv 0 \pmod{a^2}$, and using that $b^2 + c^2 \equiv 3abc \pmod{a^2}$, this is equivalent to $(n+1)abc \equiv 0 \pmod{a^2}$, which holds if $n+1 \equiv 0 \pmod a$.  Similarly, $\frac{1}{b^2}(a^2, c^2, (abc)_{n-2})$ is canonical Gorenstein if $n+1 \equiv 0 \pmod b$ and $\frac{1}{c^2}(a^2, b^2, (abc)_{n-2})$ is canonical Gorenstein if $n+1 \equiv 0 \pmod c$.  Because $a,b,c$ are relatively prime, by the Chinese Remainder Theorem, there exist infinitely many integers $n$ such that $n+1 \equiv 0 \pmod a$ and $\equiv 0 \pmod b$ and $\equiv 0 \pmod c$ (e.g. $n = k(abc) - 1$ for any integer $k > 0$). Therefore there exist infinitely many $n$ such that $X_n$ is a degeneration of $\bP^n$ with canonical Gorenstein singularities. 

    The remaining cases can similarly be turned into a system of congruences on $n$ modulo relatively prime integers and thus yield infinitely many solutions: for the second case, we find $n+1 = 0$ modulo $a,b,c$ and $2$ are sufficient; for the third we find $n = 0 \pmod 2$ and $n = 2 \pmod 3$ are sufficient; and for the last we find $n =1 \pmod 2$, $n = 2 \pmod 3$, and $n = 6 \pmod 7$ are sufficient. 
\end{proof}

\subsection{Embedding curves in degenerations of $\mathbb{P}^{g-1}$} 
As a second application, we prove that low genus hyperelliptic curves embed naturally in degenerations of $\bP^{g-1}$.

Any non-hyperelliptic smooth genus $g$ curve embeds as a degree $2g-2$ curve in $\bP^{g-1}$ via the canonical linear system (see, e.g. \cite[Proposition IV.5.2]{Hartshorne}). When $g = 3$, these curves are realized as quartics in $\bP^2$ and when $g = 4$ these curves are realized as $(2,3)$ complete intersections in $\bP^3$. 

Given a non-canonical (hyperelliptic) curve $C$ of genus $g$, it is natural to ask if $C$ admits a natural embedding into a mildly singular degeneration of $\bP^{g-1}$.  We verify this for $g = 3, 4$. 

\begin{theorem}\label{thm:hyperellipticcurves}
    For $g = 3,4$, every smooth hyperelliptic genus $g$ curve embeds in a klt degeneration $X$ of $\bP^{g-1}$ as a (weighted) complete intersection.  The divisors defining the complete intersection deform in the smoothing of $X$ to $\bP^{g-1}$, and the resulting complete intersection on $\bP^{g-1}$ is a canonical genus $g$ curve. 
\end{theorem}

\begin{proof}
    First, suppose $g = 3$.  Hyperelliptic curves of genus $3$ naturally embed as divisors of degree $8$ in $\bP(1,1,4)$ (see, e.g. \cite[Theorem 3.3]{ADLW}).  In the smoothing of $\bP(1,1,4)$ to $\bP^2$ from Example \ref{ex:VeroneseEmbeddingTrick}, the defining equation of the hyperelliptic curve deforms to an equation of degree 4 in $\bP^2$. Now, suppose $g = 4$.  The desired statement follows from Lemma \ref{lem:hyperellipticgenus4}. 
\end{proof}

For higher genus curves, one may also ask if they embed in a klt degeneration of $\bP^{g-1}$ in a natural way.  We leave this for future investigation.

\subsection{Special test configurations and K-stability}
K-stability is an algebro-geometric notion of stability for Fano varieties which has proven essential in the past decade to construct well behaved moduli spaces of Fano varieties (see, e.g. \cite{ChenyangKstabilitybook} for more details).  K-stability itself is measured by an invariant called the \textit{Futaki invariant} associated to the weight of a $\G_m$ action on an isotrivial degeneration of the given Fano variety, called a \textit{test configuration}.  

\begin{definition}
    Let $X$ be a klt Fano variety and let $L = \calO(-K_X)$.  For $r>0$ such that $L^{[r]}$ is a Cartier divisor, a \emph{normal test configuration} of index $r$ of $(X;L)$, denoted by $(\cX;\cL_r)/\bA^1$, consists of the following data:
\begin{itemize}
 \item a normal variety $\cX$ together with a flat projective morphism $\pi:\cX\to \bA^1$;
 \item a $\pi$-ample line bundle $\cL_r$ on $\cX$;
 \item a $\bG_m$-action on $(\cX;\cL_r)$ such that $\pi$ is $\bG_m$-equivariant with respect to the standard action of $\bG_m$ on $\bA^1$ via multiplication;
 \item $(\cX\setminus\cX_0;\cL_r|_{\cX\setminus\cX_0})$
 is $\bG_m$-equivariantly isomorphic to $(X;L^{[r]})\times(\bA^1\setminus\{0\})$.
\end{itemize}

A normal test configuration $(\cX;\cL_r)/\bA^1$ is called a \emph{special test configuration} if 
we have
$\cL_r\sim_{\bQ}-r(K_{\cX/\bA^1})$ and $(\cX,\cX_0)$ is plt. In this case, $\cX_0$ is necessarily a klt Fano variety.  As $\cL_r$ is uniquely determined in this case, we drop it from the notation and simply write $\cX/\bA^{1}$ or $\pi:\calX \to \bA^{1} $.  In this case, we say $\calX_0$.
\end{definition}

By \cite{LX14}, to determine the K-(semi/poly)stability of a Fano variety $X$, it is sufficient to compute the Futaki invariant of all special test configurations.  This motivates the following question: which degenerations of $\mathbb{P}^{n}$ in this article correspond to special test configurations?

\begin{corollary}
\label{cor:GmEquivariantDegenerations}
Let $\pi: \calX \to \bA^1$ be a special test configuration of a klt Fano variety $X$ and let $\calM$ be a relatively ample $\bQ$-line bundle on $\calX$ that is $\pi$-equivalent to a rational multiple of the anticanonical divisor $-K_{\calX/\bA^1}$.  Then, $\calC_{p}(\mathcal{X}/\bA^1, \mathcal{M}) \to \bA^1$ is a special test configuration of $C_p(X, \calM|_{X})$.  
\end{corollary}

\begin{proof}
Suppose that we have a special test configuration $\pi: \calX \rightarrow \mathbb{A}^{1}$ of $X$ and $\calM$ a relatively ample $\bQ$-line bundle on $\calX$ such that $\calM \sim_{\bQ}-a(K_{\cX/\bA^1})$ for some $a \in \bQ$.  Consider the relative projective cone $\calC_{p}(\mathcal{X}/\bA^1, \mathcal{M}) = \Proj_{T}\:\sum_{m \geq 0}\left(\sum_{r=0}^{m} \pi_{\ast}(\mathcal{M}^{[r]}) \cdot x_{n+1}^{m-r}\right)$.  The $\bG_m$-action on $\calX$ induces a $\bG_m$-action on $\calC_{p}(\mathcal{X}/\bA^1, \mathcal{M})$ via the given action on $\calX$ and the trivial action on $x_{n+1}$.  Furthermore, via the  $\mathbb{G}_{m}$-equivariant isomorphism
\[ \calX \times_{\bA^1} (\bA^1 - \{ 0\}) \cong X \times (\bA^1 - \{ 0\}) ,\]
we obtain a $\mathbb{G}_{m}$-equivariant isomorphism
\[ \calC_p(\calX/\bA^1,\calM) \times_{\bA^1} (\bA^1 - \{ 0\}) \cong C_p(X,\calM|_X) \times (\bA^1 - \{ 0\}). \]
For all $t \in \bA^1$, the fiber $\calC_p(\calX/\bA^1,\calM)_t \cong C_p(\calX_t, \calM_t)$ is a klt Fano variety by Lemma \ref{Lemma:relativeProjectiveCone} and hence this is a special test configuration.
\end{proof}

We apply this corollary to find infinitely many special degenerations of $\bP^n$, $n \ge 2$. 

\begin{theorem}
    For $n \ge 1$, $k \ge 0$, the weighted projective spaces $\bP(1,F_{2n-1}^2,F_{2n+1}^2,(F_{2n-1}F_{2n+1})_k)$ are special degenerations of $\bP^{k+2}$.
\end{theorem}

\begin{proof}
    By \cite[Theorem A.3]{ABB+}, the weighted projective surfaces $\bP(1,F_{2n-1}^2,F_{2n+1}^2)$ are special degenerations of $\bP^{2}$.  The result now follows immediately from Corollary \ref{cor:GmEquivariantDegenerations}.
\end{proof}

\subsection{Moduli spaces of hypersurfaces}

In this section, we apply the results in this paper to construct examples of singular varieties parametrized by the boundary of moduli spaces of either canonically polarized hypersurfaces or K-polystable Fano hypersurfaces.  As degenerations of projective space give rise to degenerations of hypersurfaces, it is natural to use the degenerations of $\bP^n$ in this paper to construct degenerations of hypersurfaces.  For example, a hypersurface $D \subset \bP^n$ of degree $d = ab$ admits a degeneration to a degree $a$ cover of a degree $b$ hypersurface as follows \cite[Example 4.3]{Mori75} via the degeneration of $\bP^n$ to $C(X_b, \calO(b))$, the cone over a hypersurface of degree $b$ (see Theorem \ref{thm:degree-d-hypersurfaces}).  The limit of the family of hypersurfaces is a degree $a$ cover of a degree $b$ hypersurface by projection away from the vertex of the cone.  In \cite{DS24}, the authors ask if the existence of certain smooth (non-hypersurface) limits of prime degree $d$ hypersurfaces correspond to the existence of klt degenerations of $\bP^n$ and study the problem for curves in $\bP^2$.  

Indeed, in the plane curve case, every smooth limit of a family of quartic plane curves is a genus 3 curve, so is either planar (and embeds in $\bP^2$) or hyperelliptic, which by Theorem \ref{thm:hyperellipticcurves} and its proof embeds in $\bP(1,1,4)$.  We see directly that every smooth limit of quartic plane curves thus embeds in a klt degeneration of $\bP^2$.  The same phenomenon persists for degree 5.  In \cite[Main Corollary 1, Example 1.11]{DS24}, it is shown that all smooth limits of plane quintics are either planar or hyperelliptic, and furthermore every hyperelliptic genus 6 curve embeds in a partial smoothing of $\bP(1,4,25)$.  

The same phenomenon was studied 
in \cite{ADL23}
for moduli of quartic surfaces in $\bP^3$, 
which are quartic K3 surfaces.  There are two other types of (smooth, algebraic) K3 surfaces: those which are hyperelliptic and those admitting an elliptic fibration.  In fact, in \cite{ADL23}, it is shown that all smooth K3 surfaces of degree 4 admit an embedding into a klt degeneration of $\bP^3$, so again the existence of particular degenerations of $\bP^3$ captures the existence of smooth (non-hypersurface) limits of these quartic surfaces.

In this section, we construct limits of families of hypersurfaces with canonical singularities contained in the weighted projective spaces that admit smoothings to $\bP^n$.  Let $d, n$ be positive integers such that $d \ne n+1$.

For $d > n+1$, consider the KSB moduli space $M_{n-1,d(d-n-1)^{n-1}}^{KSB}$ parameterizing slc canonically polarized varieties $V$ of degree $d$, dimension $n-1$ and canonical volume $(K_V)^{n-1} = d(d-n-1)^{n-1}$ (see, e.g. \cite{Kollar23}).  As a smooth degree $d > n+1$ hypersurface $V \subset \bP^n$ has dimension $n-1$ and volume $d(d-n-1)^{n-1}$ and small deformations of hypersurfaces are again hypersurfaces, there exists an irreducible component $M^{hyp}_d \subset M_{n-1,d(d-n-1)^{n-1}}^{KSB}$ generically parameterizing smooth degree $d$ hypersurfaces in $\bP^n$.  

If instead $2 < d < n+1$, we consider the K-moduli space $M_{n-1,d(n+1-d)^{n-1}}^{K}$ parameterizing K-polystable Fano varieties $V$ of dimension $n-1$ and anticanonical volume $(-K_V)^{n-1} = d(n+1-d)^{n-1}$ (see, e.g. \cite{ChenyangKstabilitybook}).  As a generic degree $2 < d < n+1$ hypersurface $V$ is K-stable (see, e.g. \cite[Corollary 1.4]{Zh21}), has dimension $n-1$ and anticanonical volume $d(n+1-d)^{n-1}$ and small deformations of hypersurfaces are again hypersurfaces, there exists an irreducible component $M^{hyp}_d \subset M_{n-1,d(n+1-d)^{n-1}}^{K}$ parameterizing K-polystable hypersurfaces and their limits.  

In both cases, we will show that divisors on the weighted projective spaces in this paper yield examples of KSB-stable and K-stable limits of hypersurfaces.  One could also use the techniques in what follows to show the existence of degenerations of Calabi-Yau hypersurfaces of degree $n+1$ in $\bP^n$ with canonical singularities, but there is no notion of projective moduli of these objects in general. 

We include a well-known proposition.

\begin{proposition}
    Let $X= \bP(a_0, \dots, a_n)$ be a weighted projective space and $m = \mathrm{lcm}(a_0, \dots, a_n)$.  For any $\ell > 0$, the linear system $|\calO_X(\ell m)|$ contains quasismooth Cartier divisors.  
\end{proposition}

\begin{proof}
    The sheaf $\calO_X(\ell m)$ is Cartier by Theorem \ref{thm:divisorsonwps}, and as $m = \mathrm{lcm}(a_0, \dots, a_n)$, we have $b_i= \ell m /a_i \in \bZ$ for each $0 \le i \le n$.  Then, the weighted hypersurface $x_0^{b_0} + \dots + x_n^{b_n}$ is quasismooth and defines a section of $\calO_X(\ell m)$.
\end{proof}

We can use this to construct limits of canonically polarized hypersurfaces contained in weighted projective spaces parameterized by the KSB moduli space.

\begin{proposition}
\label{prop:moduli-degree-d-hypersurfaces}
    For any $d,n > 0$ such that $n \ge 2$ and Markov triple $(a,b,c)$, if $abc \mid d$, there exist quasismooth weighted hypersurfaces $V \subset \bP(a^2, b^2, c^2, (abc)_{n-2})$ of weighted degree $dabc$.  Furthermore, the general such weighted hypersurface $V$ has only canonical singularities.  For $d > n+1$, these weighted hypersurfaces correspond to limits of canonically polarized hypersurfaces in the KSB moduli space $M^{hyp}_d$.  
\end{proposition}

\begin{proof}
    By Theorem \ref{thm:infiniteFamiliesDegenerationsPn}, $\bP(a^2, b^2, c^2, (abc)_{n-2})$ admits a smoothing to $\bP^n$ and, given this smoothing, there are no obstructions to extending this to a smoothing of the pair 
    \[
    (\bP(a^2, b^2, c^2, (abc)_{n-2}), V),
    \]
    where $V$ is a Cartier divisor with weighted degree $dabc$.  By the previous proposition, a generic $V$ is quasismooth, and hence has only cyclic quotient klt singularities.  To show the general such weighted hypersurface has only canonical singularities, consider the hypersurface $x_0^{dabc/a^2} + x_1^{dabc/b^2} + x_2^{dabc/c^2} + x_3^{d} \dots + x_n^{d}$.  By computation, the only singularities of this hypersurface are of the form $A_{n-1} \times W$, where $W$ is a smooth $(n-3)$-fold where $n = a,b,c$ and $A_{n-1}$ denotes the surface singularity $\frac{1}{n}(1,n-1)$, and of the form $\frac{1}{abc}(a^2,b^2,c^2) \times Z$ where $Z$ is a smooth $(n-4)$-fold, both of which are canonical by the Reid-Tai criterion.  This proves the general weighted hypersurface in this linear system has canonical singularities.  Finally, if $d > n+1$, by construction, $K_V$ is ample and $V$ deforms to a hypersurface of degree $d$, so $V$ corresponds to a point in $M^{hyp}_d$.
\end{proof}

\begin{remark}
    An analogous result holds for the other weighted projective spaces that admit smoothings to $\bP^n$, and the assumption that $abc \mid d$ can be relaxed to only the assumption that there exists a weighted hypersurface of degree $dabc$ that has log canonical singularities.  We state only the result above as the singularity condition is guaranteed if $abc \mid d$ by the previous proposition. 
\end{remark}

The previous proposition produces elements of the boundary of the KSB moduli space of hypersurfaces for infinitely many degree $d$.  We now show an analogous result for the K-moduli space of Fano hypersurfaces. 

\begin{proposition}
    For any $n \ge 2$ and Markov triple $(a,b,c)$ ordered so $a \le b \le c$, if $d \ge 2$ is such that $abc \mid d$, and $(1 - \frac{a}{bc})n + 1 < d <n+1$\footnote{Existence of $(n,d)$ satisfying these constraints will be discussed in the remarks following the proof.}, there exist K-polystable Fano weighted hypersurfaces $V \subset \bP(a^2, b^2, c^2, (abc)_{n-2})$ of weighted degree $dabc$ parametrized by the K-moduli space $M^{hyp}_d$.  Furthermore, the general such weighted hypersurface $V$ has only canonical singularities. 
\end{proposition}

\begin{proof}
    By Theorem \ref{thm:infiniteFamiliesDegenerationsPn}, $\bP(a^2, b^2, c^2, (abc)_{n-2})$ admits a smoothing to $\bP^n$ and, given this smoothing, there are no obstructions to extending this to a smoothing of the pair 
    \[
    (\bP(a^2, b^2, c^2, (abc)_{n-2}), V),
    \]
    where $V$ is a Cartier divisor with weighted degree $dabc$.  By the previous proposition, a generic $V$ is quasismooth, and hence has only cyclic quotient klt singularities. By construction, $-K_V$ is ample and $V$ deforms to a hypersurface of degree $d$, so if we can show that a general such $V$ is K-stable, it corresponds to a point of the K-moduli space $M^{hyp}_d$.  To show the general such weighted hypersurface is K-stable, consider the hypersurface $V$ given by 
    \[
    x_0^{dabc/a^2} + x_1^{dabc/b^2} + x_2^{dabc/c^2} + x_3^{d} \dots + x_n^{d} \subset \bP(a^2, b^2, c^2, (abc)_{n-2}).
    \]
    We will prove this is K-stable under the given assumptions. 

    As $abc \mid d$, write $d = k abc$, so the hypersurface $V$ is given by the vanishing of 
    \[
    x_0^{kb^2c^2} + x_1^{ka^2c^2} + x_2^{ka^2b^2} + x_3^{kabc} \dots + x_n^{kabc}.
    \]
    Consider the finite map 
    \[
    \bP(a^2, b^2, c^2, (abc)_{n-2})_{x_i} \dashrightarrow \bP^n_{y_i}
    \]
    given by 
    \[
    [x_0: x_1: x_2: x_3 : \dots : x_n] \mapsto [ x_0^{kb^2c^2}: x_1^{ka^2c^2}: x_2^{ka^2b^2}: x_3^{kabc} : \dots : x_n^{kabc}].
    \]
    This is a morphism on $V$ and the image of $V$ is the hyperplane $(\sum y_i = 0) \cong \bP^{n-1}$.  The branch locus of this morphism on $V$ is the union of the hyperplanes $H_i = \{ y_i = 0\}$ and by \cite[Theorem 1.2]{LZ22}, the K-polystability of $V$ is equivalent to the K-polystability of $(\bP^{n-1}, \Delta)$ where 
    \[
    \Delta = \Bigl( 1 - \displaystyle{\frac{1}{kb^2c^2}}\Bigr)H_0 + \Bigl( 1 - \displaystyle{\frac{1}{ka^2c^2}}\Bigr)H_1 + \Bigl( 1 - \displaystyle{\frac{1}{ka^2b^2}}\Bigr)H_2  + \sum_{i = 3}^n \Bigl(1 - \displaystyle{\frac{1}{kabc}}\Bigr) H_i.
    \]
    To verify the K-stability of this pair, we use Fujita's criterion for K-stability of hyperplane arrangements.  Writing $\Delta = \sum_{i = 0}^{n} \alpha_i H_i$ for simplicity, by \cite[Corollary 1.6]{F21}, $(\bP^{n-1}, \Delta)$ is K-stable if and only if 
    \[ \frac{1}{n} \sum_{i=0}^{n} \alpha_i > \frac{1}{\ell} \sum_{j = 1}^\ell \alpha_{i_j}\] for every $1 \le\ell  \le n-1$ and $0 \le i_{1} < \dots < i_\ell \le n$.  As the term on the right is the average of the coefficients over any set of size $\ell$, this holds if and only if the $\ell=1$ case holds for the largest coefficient $\alpha$, i.e. 
    \[ \frac{1}{n} \sum_{i=0}^{n} \alpha_i >  \alpha_{max}\] where $\alpha_{max} = \max\{ \alpha_0, \dots, \alpha_n \}$.  Assuming $a \le b \le c$, we have $\alpha_{max} = \alpha_0 = 1 - \displaystyle{\frac{1}{kb^2c^2}}$.  The pair is then K-stable if and only if 
    \[ \sum_{i=0}^n \alpha_i =  n+1 - \frac{(a^2 + b^2 + c^2 + (n-2)abc)}{ka^2b^2c^2} > n \alpha_0 = n \Bigl(1 - \frac{1}{kb^2c^2}\Bigr) \] or equivalently 
    \[ na^2 > a^2 + b^2 + c^2 + (n-2)abc - ka^2b^2c^2.\]
    Because $(a,b,c)$ is a Markov triple, $a^2 + b^2 + c^2 = 3abc$ so this is equivalent to
    \[ na^2 > (n+1)abc - ka^2b^2c^2 = (abc)(n+1-kabc)\]
    or 
    \[ n > \frac{bc}{a}(n+1-d)\] which is equivalent to 
    \[ d > \Bigl(1 - \frac{a}{bc}\Bigr) n+1\]
    Therefore, by assumption, the hyperplane arrangement $(\bP^{n-1}, \Delta)$ is K-polystable and hence $V$ is K-polystable.  

    The statement on the singularities of $V$ follows as in the previous proposition.
\end{proof}

\begin{remark}
    If $d > 2$ in the previous proposition, as $abc \mid d$, then $V$ is a Fano weighted hypersurface of degree $ \ge \max\{a^2b^2c^2, 3\}$ in $\bP(a^2,b^2,c^2,(abc)_{n-2})$.  For $(a,b,c) = (1,1,1)$, by \cite{MM63}, $\Aut(V)$ is finite, and for $(a,b,c) \ne (1,1,1)$, \cite[Theorem 3.1]{Esser24} implies that $\Aut(V)$ is finite.  Therefore, the K-polystability of $V$ implies the K-stability of $V$ (see, for example, \cite[Corollary 2.2.5]{Cheltsov+}). 
    
\end{remark}

\begin{remark}
    If $d$ is an integer and $(a,b,c)$ a Markov triple ordered such that $a\le b \le c$ such that $d = kabc$, provided $ka^2 > 1$, there exist integers $n$ such that the assumptions of the previous proposition are satisfied.  If $(a,b,c) = (1,1,1)$, the assumptions hold for any choice of $n > d-1$, so assume $(a,b,c) \ne (1,1,1)$.  Then, $(1 - \frac{a}{bc})n + 1 < d <n+1$ can be re-written as $d -1 < n < \frac{bc}{bc-a}(d-1) = d-1 + \frac{a}{bc-a}(d-1)$, and any such $n$ will suffice.  To verify existence of an integer $n$ in this range, we observe that because $ka^2 > 1$, we have $da > bc$ so $(d-1)a > bc - a$ and hence $\frac{a}{bc-a}(d-1) > 1$.  Therefore, the integer $d$ is between $d-1 < d < d-1 + \frac{a}{bc-a}(d-1)$, so we may take $n = d$.  

    Interestingly, for fixed values of $d$ and $(a,b,c)$, when $n$ grows too large (beyond the upper threshold), the Fermat-like hypersurface used in the previous proof is K-unstable.  For example, suppose $d = 4$ and $(a,b,c) = (1,1,2)$, so we are considering limits of quartic hypersurfaces on $\bP(1,1,4, 2_{n-2})$.  In the notation of the previous proof, for $n > 6$, the hypersurface $x_0^8 + x_1^8 + x_2^2 + x_3^4 + \dots + x_n^4 = 0$ is K-unstable: the associated hyperplane arrangement is $(\bP^{n-1}, \frac{7}{8} H_0 + \frac{7}{8}H_1 + \frac{1}{2} H_2 + \sum_{i=3}^{n} \frac{3}{4}H_i)$.  By Fujita's criterion above, this pair (and hence the original hypersurface) is K-semistable if and only if 
    \[ \frac{1}{n} \sum_{i=0}^{n} \alpha_i >  \alpha_{max},\] i.e.,
    \[ \frac{1}{n}\Biggl(\frac{7}{8} + \frac{7}{8} + \frac{1}{2} + (n-2) \frac{3}{4}\Biggr) \ge \frac{7}{8}\] or equivalently 
    \[ (7+7+4+(n-2))(6) \ge 7n \] or 
    \[ 6 \ge n .\]
    Therefore, this Fermat hypersurface is a K-unstable limit of quartic hypersurfaces in $\bP^n$ for $n > 6$. 
\end{remark}

\subsection{Smooth limits of complete intersections}

In the last section, we use degenerations of $\bP^4$ to construct smooth limits of families of complete intersections.  We will construct a smooth limit of a family of complete intersections of multidegree $(3,7)$ in $\bP^4$ that is not itself a complete intersection.

\begin{theorem}
    There exists a smooth family of projective surfaces $\calS \to T$ over a pointed curve such that the generic fiber $\calS_t$ is a $(3,7)$ complete intersection in $\bP^4$ but the special fiber $\calS_0$ is a smooth surface that is not a complete intersection. 
\end{theorem}

\begin{proof}
    We recall the notation and construction in Lemma \ref{lem:psofP2122149}.  Consider the embedding $\bP(2,12,21,49)_{[x_i]} \to \bP(2,2,3,7,12)_{[y_i]}$ given by 
    \[ [x_0:x_1:x_2:x_3] \mapsto [x_0^7:x_0x_1:x_2:x_3:x_1^7]\]
    with image is defined by the degree 14 equation $y_0y_4 = y_1^7$.  Let $X_{14}(t,s)$ be the degree 14 weighted hypersurface $y_0y_4 = y_1^7 + t y_3^2 + s y_2^4y_1$, which admits an isotrivial specialization to $\bP(2,12,21,49)$ sending $t, s \to 0$.

    Next, consider the embedding $\bP(2,2,3,7,12)_{[y_i]} \to \bP(1,1,3,5,6,7)_{[z_i]}$ given by 
    \[[y_0:y_1:y_2:y_3:y_4] \mapsto [y_0:y_1:y_2^2:y_2y_3:y_4:y_3^2] \]
    with image defined by $z_2z_5 = z_3^2$.  Let $Y_{7}(t,s)$ be the weighted hypersurface given by $z_0z_4 = z_1^7 + tz_5 + s z_2^2z_1$, so the image of $X_{14}(t,s)$ is given by the intersection of $Y_{7}(t,s)$ and $z_2z_5 = z_3^2$.  For $t \ne 0$, this is isomorphic to the degree 10 weighted hypersurface given by $z_2(z_0z_4 - z_1^7 - sz_2^2 z_1) = tz_3^2$ in the weighted projective space $\bP(1,1,3,5,6,7)_{[z_i]}$, and perturbing this equation by a scalar multiple of $z_1^{10}$, we have constructed a partial smoothing from $\bP(2,12,21,49)$ to the weighted hypersurface $Z(t,s,r) = \{z_2(z_0z_4 - z_1^7 - sz_2^2 z_1) = tz_3^2 + rz_1^{10} \}$.  Note that the perturbation by $rz_1^{10}$, which is exactly the hypersurface $Z(t, s, r)$, is equivalent to the intersection of $Y_{7}(t,s)$ and $z_2z_5 = z_3^2 + rz_1^{10}$ in the weighted projective space $\bP(1,1,3,5,6,7)_{[z_i]}$.  By direct computation, the only singularities of $Z(t,s,r)$ for generic $s,t,r$ are along the curve $z_0 = z_1 = z_3 = 0$ and $Z(t,s,r)$ has a $\frac{1}{3}(1,2,0)$ singularity at the generic point of this curve.  Because we've shown that the singularity is of type $A_{2} \times \mathbb{A}^{1}$ (an $A_{2}$ singularity along a curve), \Cref{prop:smoothingtocubic} applies to prove that $Z(t,s,r)$ is smoothable to a cubic threefold in $\bP^4$. 

    Now, we wish to consider the threefold $X=X_{14}(0,1) \subset \bP(2,2,3,7,12)$. Clearly $X$ is a partial smoothing of $\bP(2,12,21,49)$ and by construction is smoothable to $Z(0,1,r)$, which is smoothable to a cubic threefold.  Therefore, by Lemma \ref{lem:xtoytoz}, $X$ is smoothable to a cubic.  By construction, $X$ is a degree $14$ hypersurface in $\bP(2,2,3,7,12)_{[y_i]}$, has an (isolated) non-Gorenstein singular point at $p = [0:0:0:1:0]$, and has a reducible curve $C$ of Gorenstein cyclic quotient singularities elsewhere.   Because the curve of singularities $C$ and the point $p$ are both smoothable, we wish to show that it is possible to smooth them individually.  Smoothing the curve first will yield a partial smoothing $V$ of $X$ that is not Gorenstein whose only singularity is an isolated singular point (of the same type as $X$), and $V$ can be further smoothed to a cubic threefold. 

    To verify that they can be smoothed individually, we will show that $ H^2(X,T_X) = 0$, which will imply that we can choose any unobstructed deformation of the point (along with an unobstructed smoothing of the curve) and find a corresponding unobstructed deformation of $X$.  In particular, the smoothing of the curve coming from the smoothing of $X$ to a cubic threefold and the trivial deformation of the point will induce a deformation of $X$. 
    
    To show this let $\bP = \bP(2,2,3,7,12)$ and consider the exact sequence 
    \[ 0 \to T_X \to T_{\bP}|_X \to \calN_{X/\bP} \to \SExt^1(\Omega^{[1]}_{{X}}, \calO_{{X}}).\]
    First observe that $H^1(T_{\bP}|_X) = H^2(T_{\bP}|_X) = 0$, which follows from the exact sequence in cohomology associated to the exact sequences 
    \[ 0 \to T_{\bP}(-X) \to T_{\bP} \to T_\bP|_X \to 0,\] and \[ 0 \to \calO_{\bP} \to \calO_{\bP}(2) \oplus \calO_{\bP}(2) \oplus \calO_{\bP}(3) \oplus \calO_{\bP}(7) \oplus \calO_{\bP}(12)  \to T_\bP \to 0,\] and the vanishing $H^i(\bP, \calO_{\bP}(n)) = 0$ for all $i \ge 1, n \ge 0$ (cf. Theorem \ref{thm:cohomologyofwps}). 
    
    Because $X$ has only Gorenstein cyclic quotient singularities away from $p$, by \cite[Theorem 3.3.10]{BH98}, $\SExt^1(\Omega^{[1]}_{\tilde{X}}, \calO_X)$ is $0$ away from $p$.  To show $H^2(X,T_X) = 0$, denoting by $\calQ = \im (\calN_{X/\bP} \to \SExt^1(\Omega^{[1]}_{{X}}, \calO_{{X}}))$, it suffices to show that there is a surjection on global sections $H^0(X, \calN_{X/\bP}) \to H^0(X, \calQ)$.  Indeed, letting $\calK = \ker (\calN_{X/\bP} \to \calQ)$, this implies $H^1(X, \calK) = 0$ from the long exact sequence in cohomology, but exactness of 
    \[ 0 \to T_X \to T_{\bP}|_X \to \calK \to 0 \] and the vanishing $H^1(T_{\bP}|_X) = H^2(T_{\bP}|_X) = 0$ implies $H^2(X,T_X) = H^1(X, \calK) = 0$.  Therefore, we need to verify the surjection $H^0(X, \calN_{X/\bP}) \to H^0(X, \calQ)$.
    
    As $\calQ$ is supported only on $p$, we work locally in the affine open neighborhood of $p$ $U_p = (y_3 \ne 0) \subset \bP=\bP(2,2,3,7,12)$.  By definition, $U_p$ is the quotient $\bA^4/\mu_7 = \frac{1}{7}(2,2,3,12)$ and on this neighborhood, $X$ is given by the equations $X = (y_0y_4 - y_1^7 - y_2^4y_1 = 0)$, so $X$ is in fact a Cartier divisor.  Let $\tilde{X}$ be given by the same equation on $\bA^4_{(y_0,y_1,y_2,y_4)}$.  Working on the affine space, there is an exact sequence
    \[ 0 \to T_{\tilde{X}} \to T_{\bA^4}|_{\tilde{X}} \to \calN_{\tilde{X}/\bA^4} \to \SExt^1(\Omega^{1}_{\tilde{X}}, \calO_{\tilde{X}}) \]
    whose $\mu_7$-invariants induces the corresponding exact sequence (see, e.g. \cite[Appendix B]{GKKP}, using the fact that $X$ is Cartier near $p$) near $p$
    \[ 0 \to T_{X} \to T_{\bP}|_{{X}} \to \calN_{{X}/\bP} \to \calQ . \]
    Therefore, it suffices to show that each $\mu_7$-invariant element of $\SExt^1(\Omega^{1}_{\tilde{X}}, \calO_{\tilde{X}})$ comes from an element of $H^0(X, \calN_{X/\bP})$.  This calculation can be done directly: because $\tilde{X} = (f = 0)$ is a hypersurface in $\bA^4$, the elements of $\SExt^1(\Omega^{1}_{\tilde{X}}, \calO_{\tilde{X}})$ are the generators of the Jacobian algebra $\bC[y_0,y_1,y_2,y_4]/(f,\partial f/\partial y_0, \partial f/\partial y_1, \partial f/\partial y_2, \partial f/\partial y_4)$, which we compute directly and find the $\mu_7$-invariant elements to be $\{ 1, y_1^2 y_2, y_1^4y_2^2\}$.  These are the precisely the images of the elements of $\{ y_3^2, y_1^2y_2y_3, y_1^4 y_2^2\} \in H^0(X,\calN_{X/\bP})$ (which represent embedded deformations of the degree $7$ hypersurface in $\bP$).  Therefore, $H^0(X, \calN_{X/\bP})$ surjects onto $H^0(X,\calQ)$ and therefore $H^2(X,T_X) = 0$. 
    
    Now, define the surface $\calS_0$ to be a smooth section of the line bundle $L$ that is the deformation of $\calO(294)$ on $\bP(2,12,21,49)$.  Because $\calO(294)$ has degree $7 \cdot 42$, in the deformation of $\bP(2,12,21,49)$ to the cubic threefold, it deforms to $\calO(7)$ (as $\calO(42)$ deforms to $\calO(1)$).  Because a generic section of $\calO(294)$ on $\bP(2,12,21,49)$ misses the singular locus $[0:0:x_2:x_3]$, the generic section of the line bundle $L$ on $V$ misses the singular point of $V$ and will be smooth.  Indeed, $L$ is very ample as it is the deformation of a very ample sheaf $\calO(294)$ on $\bP(2,12,21,49)$, and $H^i(\bP(2,12,21,49), \calO(294m)) = 0$ for all $i, m > 0$ by \cite[8.38]{Kollar23}, so by Bertini's Theorem, a generic section $\calS_0$ of $L$ on $V$ is smooth. 

    Furthermore, because $V$ is smoothable to a cubic and $L$ deforms to $\calO(7)$ in this deformation, $\calS_0$ deforms to a smooth $(3,7)$ complete intersection in $\bP^4$.  Let $\calS$ be the corresponding family of surfaces.

    Finally, as the pair $(V,\calS_0)$ is plt by construction, we claim that $\calS_0$ cannot be a $(3,7)$ complete intersection.  If it were, then $\calS$ could be realized in a family $\calX$ of cubic threefolds, and as $\calS$ is a smooth family, all of the threefolds must have isolated singularities and hence must be klt.  However, this implies $(\calX_0, \calS_0)$ and $(V,\calS_0)$ are both limits of the same family, but $V$ is not Gorenstein so cannot be a cubic hypersurface, and therefore this would contradict the separatedness of moduli of canonically polarized varieties. 
\end{proof}

We expect that the degenerations of $\bP^n$ constructed in this paper and their partial smoothings can be used to construct many families of smooth limits of complete intersections.

\bibliography{wps.bib}{}
\bibliographystyle{alpha}

\end{document}